\documentclass[11pt]{amsproc}

\usepackage{mathrsfs}

\newcommand\R{\mathbb{R}}
\newcommand\Z{\mathbb{Z}}
\newcommand\C{\mathbb{C}}
\newcommand\diff{\mathrm{d}}
\newcommand\ii{\mathrm{i}}
\newcommand{\T}{\mathbb{T}}
\newcommand{\vol}{\mathrm{vol}}
\newcommand{\KCM}{\mathrm{KCM}}

\newtheorem{theorem}{Theorem}[section]

\newtheorem{proposition}[theorem]{Proposition}
\newtheorem{corollary}[theorem]{Corollary}
\newtheorem{conjecture}[theorem]{Conjecture}

\theoremstyle{definition}
\newtheorem{definition}[theorem]{Definition}
\newtheorem{example}[theorem]{Example}

\newtheorem{problem}[theorem]{Problem}

\theoremstyle{remark}

\numberwithin{equation}{section}

\begin{document}

\title{Sasaki-Einstein Manifolds}

\author{James Sparks}
\address{Mathematical Institute, University of Oxford, 24-29 St Giles', Oxford OX1 3LB, U.K.}
%\curraddr{Mathematical Institute, University of Oxford, 24-29 St Giles', Oxford OX1 3LB, U.K.}
\email{sparks@maths.ox.ac.uk}
\thanks{The author is supported by a Royal Society University Research Fellowship.}

%    General info
\subjclass{Primary 53C25, 53C15}
\date{January 20, 2010.}

\keywords{Sasakian geometry, Sasaki-Einstein manifolds}

% Abstract: This article is an overview of some of the remarkable progress that has been made in Sasaki-Einstein geometry over the last decade, 
% which includes a number of new methods of constructing Sasaki-Einstein manifolds and obstructions.

% Invited contribution to Surveys in Differential Geometry.

\maketitle

A Sasaki-Einstein manifold is a Riemannian manifold $(S,g)$ that is both Sasakian and Einstein. 

Sasakian geometry is the odd-dimensional cousin of K\"ahler geometry. Indeed, just as K\"ahler geometry is 
the natural intersection of complex, symplectic, and Riemannian geometry, so Sasakian geometry is the natural intersection of 
CR, contact, and Riemannian geometry. Perhaps the most straightforward definition is the following: a Riemannian manifold 
$(S,g)$ is Sasakian if and only if its metric cone $\left(C(S)=\R_{>0}\times S,\bar{g}=\diff r^2 + r^2 g\right)$ is K\"ahler. 
In particular, $(S,g)$ has odd dimension $2n-1$, where $n$ is the complex dimension of the K\"ahler cone. 

A metric $g$ is Einstein if $\mathrm{Ric}_g=\lambda g$ for some constant $\lambda$. It turns out that a Sasakian manifold 
can be Einstein only for $\lambda=2(n-1)$, so that $g$ has positive Ricci curvature. 
Assuming, as we shall do throughout, that $(S,g)$ is complete, it follows 
from Myers' Theorem that $S$ is compact with finite fundamental group. Moreover, 
a simple calculation shows that a Sasakian metric $g$ is Einstein with $\mathrm{Ric}_g=2(n-1)g$ if and only 
if the cone metric $\bar{g}$ is Ricci-flat, $\mathrm{Ric}_{\bar{g}}=0$. It immediately follows that for a Sasaki-Einstein manifold the 
restricted holonomy group of the cone $\mathrm{Hol}^0(\bar{g})\subset SU(n)$. 

The canonical example of a Sasaki-Einstein manifold is the odd dimensional sphere $S^{2n-1}$, equipped 
with its standard Einstein metric. In this case the K\"ahler cone is $\C^n\setminus\{0\}$, equipped with its flat metric.

A Sasakian manifold $(S,g)$ inherits a number of geometric structures from the K\"ahler structure of its cone. In particular, 
an important role is played by the Reeb vector field. This may be defined as $\xi=J(r\partial_r)$, where $J$ denotes 
the integrable complex structure of the K\"ahler cone. The restriction of $\xi$ to $S=\{r=1\}=\{1\}\times S\subset C(S)$ 
is a unit length Killing vector field, and its orbits thus define a one-dimensional foliation of $S$ called the Reeb foliation. 
There is then a classification of Sasakian manifolds, and hence also Sasaki-Einstein manifolds, according to the global 
properties of this foliation. If all the orbits of $\xi$ are compact, and hence circles, then $\xi$ integrates to a locally free 
isometric action of $U(1)$ on $(S,g)$. If this action is free, then the Sasakian manifold is said to be regular; otherwise it is said
to be quasi-regular. On the other hand, if $\xi$ has a non-compact orbit the Sasakian manifold is said to be irregular.

For the purposes of this introduction it is convenient to focus on the case of dimension 5 ($n=3$), since this is the lowest 
non-trivial dimension. Prior to the 21st century, the only known examples of Sasaki-Einstein 5-manifolds were regular. As we shall explain, a regular 
Sasaki-Einstein manifold is the total space of a principal $U(1)$ bundle over a K\"ahler-Einstein manifold of positive 
Ricci curvature. The classification of such Fano K\"ahler-Einstein surfaces due to Tian-Yau then leads to a classification 
of all regular Sasaki-Einstein 5-manifolds. Passing to their simply-connected covers, these 
are connected sums $S^5 \# k(S^2\times S^3)$ where $k=0,1,3,4,5,6,7,8$. For each of $k=0,1,3,4$ there is a unique such regular Sasaki-Einstein structure, while for $5\leq k\leq 8$ there 
are continuous families of complex dimension $2(k-4)$. However, before 2001 it was not known whether 
quasi-regular Sasaki-Einstein 5-manifolds existed, and indeed there was even a conjecture
that irregular Sasaki-Einstein manifolds do not exist. 

The progress over the last decade has been dramatic. Again, focusing on dimension 5, it is now known that there exist 
Sasaki-Einstein structures on $\# k(S^2\times S^3)$ for all values of $k$. These include infinitely many toric 
Sasaki-Einstein metrics, meaning that the torus $\mathbb{T}^3$ acts isometrically on the Sasakian structure, 
for every value of   $k$. Indeed, for $k=1$ these metrics are known completely explicitly, giving countably infinite families of 
quasi-regular and irregular Sasaki-Einstein structures on $S^2\times S^3$. The list of Sasaki-Einstein structures on $\# k(S^2\times S^3)$ also includes 
examples with the smallest possible isometry group, namely the $\mathbb{T}\cong U(1)$ generated by a quasi-regular Reeb vector field. In particular, 
there are known to exist infinitely many such structures for every $k$, except $k=1,2$, and these often come 
in continuous families.  Moreover, $S^5$ itself admits at least 80 inequivalent quasi-regular Sasaki-Einstein structures. Again, 
some of these come in continuous families, the largest known having complex dimension 5. There are also quasi-regular Sasaki-Einstein metrics
on 5-manifolds which are not connected sums of $S^2\times S^3$, including infinitely many rational homology 5-spheres, as well as infinitely 
many connected sums of these. Similar abundant results hold also in higher dimensions. In particular, all 28 oriented
diffeomorphism classes on $S^7$ admit Sasaki-Einstein metrics. 

In this article I will review these developments, as well as 
other results not mentioned above. 
It is important to mention those topics that will not be covered. As in K\"ahler geometry, one can 
define extremal Sasakian metrics and Sasaki-Ricci solitons. Both of these generalize the notion of a Sasaki-Einstein manifold in different directions. 
Perhaps the main reason for not discussing these topics, other than reasons of time and space, 
is that so far there are not any obvious applications of these geometries to supergravity theories and string theory, which is the author's main interest. 

The article is arranged as follows. Section \ref{sec:Sas} is a review of all the necessary background in Sasakian geometry; section \ref{sec:regular} covers regular Sasaki-Einstein manifolds; section \ref{sec:quasi} 
describes the construction of quasi-regular Sasaki-Einstein structures, focusing in particular on links of weighted homogeneous hypersurface singularities; in section \ref{sec:explicit} we describe what is known about 
explicit constructions of Sasaki-Einstein manifolds; section \ref{sec:toric} covers toric Sasakian geometry and the classification of toric Sasaki-Einstein manifolds; section \ref{sec:obstructions} describes obstructions; and section
\ref{sec:outlook} concludes with some open problems.

\

\subsubsection*{Acknowledgements} I would like to thank Charles Boyer for comments on the first version of this article.

\section{Sasakian geometry}\label{sec:Sas}

\subsection{Sasakian basics}\label{sec:basics}

We begin with an introduction to Sasakian geometry. Many of the results here are elementary and follow almost immediately from the definitions. 
The reader is referred to \cite{BG99,  FOW09, MSY07} or the recent monograph \cite{BG07} for detailed proofs. 

\begin{definition}
A compact Riemannian manifold $(S,g)$ is \emph{Sasakian} if and only if its metric cone
$\left(C(S)=\R_{>0}\times S, \bar{g}=\diff r^2 + r^2 g\right)$ is K\"ahler. 
\end{definition}
It follows that $S$ has odd dimension $2n-1$, where $n$ denotes the complex dimension of the 
K\"ahler cone. Notice that the Sasakian manifold $(S,g)$ is naturally isometrically embedded into the cone via 
the inclusion $S=\{r=1\}=\{1\}\times S\subset C(S)$. We shall often regard $S$ as embedded into $C(S)$ this way. There is also a canonical projection 
$p:C(S)\rightarrow S$ which forgets the $r$ coordinate. 
Being K\"ahler, the cone $(C(S),\bar{g})$ is equipped with 
an integrable complex structure $J$ and a K\"ahler 2-form $\omega$, both of which are parallel with respect 
to the Levi-Civita connection $\bar{\nabla}$ of $\bar{g}$. 
The K\"ahler structure of $(C(S),\bar{g})$, combined with its cone structure, induces the Sasakian structure 
on $S=\{1\}\times S\subset C(S)$. 

The following equations
are useful in proving many of the formulae that follow:
\begin{eqnarray}\label{rels}
&&\bar{\nabla}_{r\partial_r}\left( r\partial_r\right) = r\partial_r~, \qquad \bar{\nabla}_{r\partial_r} X = \bar{\nabla}_X \left(
r\partial_r\right) = X~,\\
&& \bar{\nabla}_X Y = \nabla_X Y -{g}(X,Y) r\partial_r~.\nonumber
\end{eqnarray}
Here $X$ and $Y$ denote vector fields on $S$, appropriately interpreted also as vector fields on $C(S)$, 
and ${\nabla}$ is the Levi-Civita connection of $g$, . 

The canonical vector field $r\partial_r$ is known as the \emph{homothetic} or \emph{Euler vector field}. 
Using the relations (\ref{rels}), together with the fact that $J$ is parallel, $\bar{\nabla}J=0$, 
one easily shows that $r\partial_r$ is real holomorphic, $\mathcal{L}_{r\partial_r} J = 0$. It is then  
natural to define the \emph{characteristic vector field}
\begin{equation}\label{Reeb}
 \xi = J\left(r\partial_r\right)~.
\end{equation}
Again, elementary calculations show that $\xi$ is real holomorphic and also Killing, $\mathcal{L}_\xi \bar{g}=0$. Moreover, 
$\xi$ is clearly tangent to surfaces of constant $r$ and has square length $\bar{g}(\xi,\xi)=r^2$. 

We may similarly define the 1-form
\begin{equation}\label{etadef}
 \eta = \diff^c \log r = \ii (\bar{\partial}-\partial) \log r~,
\end{equation}
where as usual $\diff^c = J\circ \diff$ denotes the composition of exterior derivative with the action of $J$ on 
1-forms, and $\partial$, $\bar{\partial}$ are the usual Dolbeault operators, with $\diff = \partial + \bar{\partial}$. 
It follows straightforwardly from the definition that
\begin{equation}\label{etaReeb}
 \eta(\xi) = 1, \qquad i_\xi\diff\eta=0~.
\end{equation}
Here we have introduced the interior contraction: if $\alpha$ is a $(p+1)$-form and $X$ a vector field
then $i_X\alpha$ is the $p$-form defined via $i_X\alpha(X_1,\ldots,X_p) = 
\alpha(X,X_1,\ldots,X_p)$. Moreover, it is also clear that 
\begin{equation}\label{eta}
\eta(X) = \frac{1}{r^2}\bar{g}\left(J(r\partial_r),X\right) = \frac{1}{r^2}\bar{g}(\xi,X)~.
\end{equation}
Using this last formula one can show that the K\"ahler 2-form on $C(S)$ is 
\begin{equation}\label{kahlerform}
\omega = \frac{1}{2}\diff(r^2\eta) = \frac{1}{2}\ii \partial\bar\partial r^2~.
\end{equation}
The function $\tfrac{1}{2}r^2$ is hence a global K\"ahler potential for the cone metric.

The 1-form $\eta$ restricts to a 1-form $\eta\mid_S$ on $S\subset C(S)$. One checks from $\mathcal{L}_{r\partial_r}\eta=0$ that in fact 
$\eta=  p^*(\eta\mid_S)$. In a standard abuse of notation, we shall then not distinguish between 
the 1-form $\eta$ on the cone and its restriction to the Sasakian manifold $\eta\mid_S$. 
Similar remarks 
apply to the Reeb vector field $\xi$: by the above comments this is tangent to $S$, where it defines a unit length Killing 
vector field, so $g(\xi,\xi)=1$ and $\mathcal{L}_\xi g=0$. Notice from (\ref{eta}) that $\eta(X) = g(\xi,X)$ holds for all vector fields $X$ on $S$. 
 
Since the K\"ahler 2-form $\omega$ is in particular symplectic, it follows from (\ref{kahlerform}) that the top degree 
form $\eta\wedge (\diff \eta)^{n-1}$ on $S$ is nowhere zero; that is, it is a volume form on $S$. 
By definition, this makes 
$\eta$ a \emph{contact 1-form} on $S$. Indeed, the open symplectic manifold $\left(C(S)=\R_{>0}\times S,\omega=\tfrac{1}{2}\diff(r^2\eta)\right)$ is called the \emph{symplectization} of the contact manifold $(S,\eta)$.
The relations $\eta(\xi)=1$, $i_\xi\diff\eta=0$  from (\ref{etaReeb})
imply that $\xi$ is the unique \emph{Reeb vector field} for this contact structure. We shall hence also 
refer to $\xi$ as the Reeb vector field of the Sasakian structure. 

The contact subbundle $D\subset TS$ is defined as $D=\ker \eta$. The tangent 
bundle of $S$ then splits as
\begin{equation}\label{split}
TS=D\oplus L_\xi~,
\end{equation}
where $L_\xi$ denotes the line tangent to $\xi$. This splitting is easily seen to be orthogonal 
with respect to the Sasakian metric $g$. 

Next define a section $\Phi$ of 
$\mathrm{End}(TS)$ via $\Phi\mid_D = J\mid_D$, $\Phi\mid_{L_\xi}=0$. Using $J^2=-1$ and that the cone metric $\bar{g}$ is Hermitian one shows that
\begin{eqnarray}\label{phisq}
 \Phi^2&=&-1 + \eta\otimes\xi~,\\\label{phicomplex}
g\left(\Phi(X),\Phi(Y)\right) &= &g(X,Y) - \eta(X)\eta(Y)~,
\end{eqnarray}
where $X$, $Y$ are any vector fields on $S$. In fact a triple $(\eta,\xi,\Phi)$, with $\eta$ a contact 1-form with 
Reeb vector field $\xi$ and $\Phi$ a section of $\mathrm{End}(TS)$ satisfying (\ref{phisq}), is known as  
an \emph{almost contact structure}. This implies that 
$(D,J_D\equiv\Phi\mid_D)$ defines 
an almost CR structure. Of course, this has been induced by embedding $S$ as a real hypersurface 
$\{r=1\}$ in a complex manifold, so this almost CR structure is of hypersurface type (the bundle $D$ has rank one less than that of $TS$)
and is also integrable. The Levi form may be taken to be $\frac{1}{2}\ii\bar{\partial}\partial r^2$, which is the K\"ahler form of the cone (\ref{kahlerform}). 
Since this is positive, by definition we have a strictly pseudo-convex CR structure. 
The second equation (\ref{phicomplex}) then says that $g\mid_D$ is a Hermitian metric on $D$. 
Indeed, an almost contact structure $(\eta,\xi,\Phi)$ together with a metric $g$ satisfying 
(\ref{phicomplex}) is known as a \emph{metric contact structure}. Sasakian manifolds are thus 
special types of metric contact structures, which is how Sasaki originally introduced them \cite{Sas60}. 
Since
\begin{equation}
 g(X,Y) = \frac{1}{2}\diff\eta(X,\Phi(Y)) + \eta(X)\eta(Y)~,\nonumber
\end{equation}
we see that $\frac{1}{2}\diff\eta\mid_D$ is the fundamental 2-form associated to $g\mid_D$. The contact subbundle $D$ is 
symplectic with respect to this 2-form.

The tensor $\Phi$ may also be defined via
\begin{equation}
 \Phi(X)=\nabla_X\xi~.\nonumber
\end{equation}
This follows from the last equation in (\ref{rels}), together with the calculation $\bar{\nabla}_X\xi = \bar{\nabla}_X\left(J\left(r\partial_r\right)\right) = 
J(X)$. Then a further calculation gives
\begin{equation}\label{derPhi}
 \left(\nabla_X\Phi\right)Y = g(\xi,Y)X - g(X,Y)\xi~,
\end{equation}
where $X$ and $Y$ are any vector fields on $S$. This leads to the following equivalent definitions 
of a Sasakian manifold \cite{BG99}, the first of which is perhaps the closest to the original definition of 
Sasaki \cite{Sas60}.
\begin{proposition}\label{prop:equiv}
 Let $(S,g)$ be a Riemannian manifold, with $\nabla$ the Levi-Civita connection of $g$ and $R(X,Y)$ the Riemann 
curvature tensor. Then the following are equivalent:
\begin{enumerate}
 \item There exists a Killing vector field $\xi$ of unit length so that the tensor field 
 $\Phi(X)=\nabla_X\xi$ satisfies (\ref{derPhi}) 
for any pair of vector fields $X$, $Y$ on $S$.
 \item There exists a Killing vector field $\xi$ of unit length so that the Riemann 
curvature satisfies
\begin{equation}
 R(X,\xi)Y = g(\xi,Y)X - g(X,Y)\xi~,\nonumber
\end{equation}
for any pair of vector fields $X$, $Y$ on $S$.
\item The metric cone $(C(S),\bar{g})=(\R_{>0}\times S,\diff r^2+r^2 g)$ over $S$ is K\"ahler.
\end{enumerate}
\end{proposition}
The equivalence of (1) and (2) follows from an elementary calculation relating 
$(\nabla_X\Phi)Y$ to $R(X,\xi)Y$. We have already sketched the proof that (3) implies (1). 
To show that (1) implies (3) one defines an almost complex structure $J$ on $C(S)$ via
\begin{equation}
 J\left(r\partial_r\right) = \xi~,\quad J(X) = \Phi(X) - \eta(X)r\partial_r~,\nonumber
\end{equation}
where $X$ is a vector field on $S$, appropriately interpreted as a vector field on $C(S)$, 
and $\eta(X)=g(\xi,X)$. It is then straightforward to check from the definitions that the cone is indeed K\"ahler.

We may think of a Sasakian manifold as the collection $\mathcal{S}=(S,g,\eta,\xi,\Phi)$. As mentioned, this is a special type 
of metric contact structure.

\subsection{The Reeb foliation}

The Reeb vector field $\xi$ has unit length on $(S,g)$ and in particular is nowhere zero. 
Its integral curves are geodesics, and the corresponding foliation $\mathcal{F}_\xi$ is 
called the \emph{Reeb foliation}. Notice that, due to the orthogonal splitting (\ref{split}), 
the contact subbundle $D$ is the normal bundle to $\mathcal{F}_\xi$. The leaf space is clearly identical to that 
of the complex vector field $\xi-\ii J(\xi) = \xi + \ii r\partial_r$ on the cone 
$C(S)$. Since this complex vector field is holomorphic, the Reeb foliation thus 
naturally inherits a transverse holomorphic structure. In fact the leaf space 
also inherits a K\"ahler metric, giving rise to a transversely K\"ahler foliation, as we now describe. 

Introduce a foliation chart $\{U_\alpha\}$ on $S$, where each 
$U_\alpha$ is of the form $U_\alpha = I\times V_\alpha$ with 
$I\subset \R$ an open interval and $V_\alpha\subset \C^{n-1}$ open. 
We may introduce coordinates  $(x,z_1,\ldots,z_{n-1})$ on $U_\alpha$,
where $\xi=\partial_x$ and $z_1,\ldots,z_{n-1}$ are complex coordinates 
on $V_\alpha$. The fact that the cone is complex implies that the 
transition functions between the $V_\alpha$ are holomorphic. More precisely, 
if $U_\beta$ has corresponding coordinates $(y,w_1,\ldots,w_{n-1})$, with 
$U_\alpha\cap U_\beta\neq \emptyset$, then
\begin{equation}
\frac{\partial z_i}{\partial \bar{w_j}}=0~,\qquad \frac{\partial z_i}{\partial y}=0~.\nonumber 
\end{equation}
Recall that the contact subbundle $D$ is equipped with the almost complex structure $J_D$, so that 
on $D\otimes\C$ we may define the $\pm \ii$ eigenspaces of $J_D$ as the $(1,0)$ and $(0,1)$ vectors, 
respectively. Then in the above foliation chart $U_\alpha$, 
$\left(D\otimes \C\right)^{(1,0)}$ is spanned by
\begin{equation}
\partial_{z_i} - \eta\left(\partial_{z_i}\right)\xi~.\nonumber
\end{equation}

Since $\xi$ is a Killing vector field, and so preserves the metric $g$, it follows that 
$g\mid_D$ gives a well-defined Hermitian metric $g_\alpha^T$ on the patch $V_\alpha$ by 
restricting to a slice $\{x=\mathrm{constant}\}$. Moreover, (\ref{etaReeb}) implies that
\begin{equation}
 \diff\eta\left(\partial_{z_i}-\eta\left(\partial_{z_i}\right)\xi,\partial_{\bar{z}_j}-\eta\left(\partial_{\bar{z}_j}\right)\xi\right) = 
\diff\eta\left(\partial_{z_i},\partial_{\bar{z}_j}\right)~.\nonumber
\end{equation}
The fundamental 2-form $\omega_\alpha^T$ for the Hermtian metric $g_\alpha^T$ in the patch $V_\alpha$ 
is hence equal to the restriction of $\tfrac{1}{2}\diff\eta$ to a slice $\{x=\mathrm{constant}\}$.  
Thus $\omega_\alpha^T$ is closed, and the transverse metric $g_\alpha^T$ is K\"ahler. 
Indeed, in the chart $U_\alpha$ notice that we may write 
\begin{equation}
\eta = \diff x + \ii \sum_{i=1}^{n-1}\partial_{z_i} K_\alpha \diff z_i - \ii \sum_{i=1}^{n-1}\partial_{\bar{z}_{i}}K_\alpha \diff\bar{z}_{i}~,
\nonumber
\end{equation}
where $K_\alpha$ is a function on $U_\alpha$ with $\partial_x K_\alpha=0$. The local function $K_\alpha$ is a 
K\"ahler potential for the transverse K\"ahler structure in the chart $U_\alpha$, as observed in \cite{GKN00}.
Such a structure is called a \emph{transversely K\"ahler foliation}. We denote 
the collection of transverse metrics by $g^T=\{g_{\alpha}^T\}$. Although 
$g^T$ so defined is really a collection of metrics in each coordinate chart, 
notice that we may identify it with the global tensor field on $S$ defined via
\begin{equation}
 g^T(X,Y) = \frac{1}{2}\diff\eta(X,\Phi(Y))~.\nonumber
\end{equation}
That is, the restriction of $g^T$ to a slice $\{x=\mathrm{constant}\}$ in the patch 
$U_\alpha$ is equal to $g_\alpha^T$. Similarly, the transverse K\"ahler form $\omega^T$ may be 
defined globally as $\tfrac{1}{2}\diff\eta$. 

The \emph{basic forms} and \emph{basic cohomology}  of the Reeb foliation $\mathcal{F}_\xi$ play an important role
(for further background, the reader might consult \cite{Ton97}):
\begin{definition}
A $p$-form $\alpha$ on $S$ is called \emph{basic} if 
\begin{equation}
i_\xi\alpha=0~,\qquad \mathcal{L}_\xi \alpha=0~.\nonumber
\end{equation}
\end{definition}
We denote by $\Lambda_B^p$ the sheaf of germs of basic $p$-forms and $\Omega_B^p$ the set of global 
sections of $\Lambda_B^p$. 

If $\alpha$ is a basic form then it is easy to see that $\diff\alpha$ is also basic. We may thus define 
$\diff_B=\diff\mid_{\Omega^*_B}$, so that $\diff_B:\Omega^p_B\rightarrow \Omega_B^{p+1}$. 
The corresponding complex $(\Omega^*_B,\diff_B)$ is called the \emph{basic de Rham complex}, 
and its cohomology $H_B^*(\mathcal{F}_\xi)$ the \emph{basic cohomology}.

Let $U_\alpha$ and $U_\beta$ be coordinate patches as above, with coordinates adapted to the Reeb foliation
$(x,z_1,\ldots,z_{n-1})$ and $(y,w_1,\ldots,w_{n-1})$, respectively. 
Then a form of Hodge type $(p,q)$ on $U_\alpha$
\begin{equation}
\alpha=\alpha_{i_1\cdots i_p \bar{j}_1\cdots \bar{j}_q} \diff z_{i_1}\wedge\cdots \wedge\diff z_{i_p}
\wedge\diff\bar{z}_{\bar{j}_1}\wedge\cdots\wedge\diff \bar{z}_{\bar{j}_q}~,\nonumber
\end{equation}
is also of type $(p,q)$ with respect to the $w_i$ coordinates on $U_\beta$. Moreover, if $\alpha$ is basic then 
$\alpha_{i_1\cdots i_p\bar{j}_1\cdots \bar{j}_q}$ is independent of $x$. 
There are hence globally well-defined Dolbeault operators
\begin{eqnarray}
\partial_B &:& \Omega^{p,q}_B \rightarrow \Omega^{p+1,q}_B~,\nonumber\\
\bar\partial_B &:& \Omega^{p,q}_B \rightarrow \Omega^{p,q+1}_B~.\nonumber
\end{eqnarray}
Both are nilpotent of degree 2, so that one can define the \emph{basic Dolbeault complex} $(\Omega^{p,*},\bar\partial_B)$ 
and corresponding cohomology groups $H^{p,*}_B(\mathcal{F}_\xi)$. These invariants of the Reeb foliation are important 
invariants of the Sasakian manifold.
Clearly $\diff_B=\partial_B+\bar{\partial}_B$, and we may similarly define the operator
$\diff^c_B=\ii (\bar{\partial}_B-\partial_B)$. 

In each local leaf space $V_\alpha$ we may define the Ricci forms as $\rho_\alpha^T=-\ii\partial\bar\partial 
\log \det g_\alpha^T$. These are the $(1,1)$-forms associated to the Ricci tensors 
$\mathrm{Ric}^T_\alpha$ of $g_\alpha^T$ via the complex structure, in the usual way. Via pull-back 
to $U_\alpha$ under the projection $U_\alpha\rightarrow V_\alpha$ that identifies points on the same Reeb orbit, 
one sees that these patch together to give global tensors $\rho^T$ and $\mathrm{Ric}^T$ on $S$, just as 
was the case for the transverse K\"ahler form and metric. In particular, the \emph{transverse Ricci form} 
$\rho^T$ is a global basic 2-form of Hodge type $(1,1)$ which is closed under $\diff_B$. The corresponding 
basic cohomology class $[\rho^T/2\pi]$ is denoted by $c_1^B(\mathcal{F}_\xi)\in H_B^{1,1}(\mathcal{F}_\xi)$, or simply $c_1^B$, and is called the \emph{basic first Chern class of $\mathcal{F}_\xi$}. 
Again, this is an important invariant of the Sasakian structure. We say that $c_1^B>0$ or $c_1^B<0$ if $c_1^B$ or $-c_1^B$ is represented by a transverse K\"ahler form, respectively. In particular, the Sasakian structure will be called \emph{transverse Fano} 
if $c_1^B>0$.

Thus far we have considered a fixed Sasakian manifold $\mathcal{S}=(S,g,\eta,\xi,\Phi)$. It will be important later 
to understand how one can deform such a structure to another Sasakian structure on the same manifold $S$. 
Since a Sasakian manifold and its corresponding K\"ahler cone have several geometric structures, 
one can fix some of these whilst deforming others. An important class of such deformations, 
which are analogous to deformations in K\"ahler geometry, is summarized by the following result:
\begin{proposition}\label{prop:def}
Fix a Sasakian manifold $\mathcal{S}=(S,g,\eta,\xi,\Phi)$. Then any other Sasakian structure on $S$ with the same 
Reeb vector field $\xi$, the same holomorphic structure on the cone $C(S)=\R_{>0}\times S$, and the same 
transversely holomorphic structure of the Reeb foliation $\mathcal{F}_\xi$ is related to the original structure via
the deformed contact form $\eta'=\eta+\diff_B^c\phi$, where $\phi$ is a smooth basic 
function that is sufficiently small. 
\end{proposition}
\begin{proof}
The proof is straightforward. We fix the holomorphic structure on the cone, but 
replace $r$ by ${r}'=r\exp\phi$, where $\tfrac{1}{2}\left(r'\right)^2$ will be the new global K\"ahler potential 
for the cone metric. Since $J$ and $\xi$ are held fixed, from (\ref{Reeb}) we have ${r}'\partial_{{r}'}=
r\partial_r$, which implies that $\phi$ is a function on $S$. The new metric on $C(S)$ will be
${\bar{g}'}(X,Y)={\omega'}(X,JY)$ where ${\omega}'=\tfrac{1}{2}\ii 
\partial\bar{\partial}{r'}^2$. Since $r\partial_r$ is holomorphic, it is clear that ${\bar{g}'}$
 is homogeneous degree 2. Moreover, one easily checks that the necessary condition $\mathcal{L}_\xi{r}'=0$, which 
means that $\phi$ is basic, implies that ${\bar{g}'}$ is also a cone. Then
\begin{equation}
{\eta}' = \diff^c \log {r}' = \eta + \diff_B^c \phi~.\nonumber
\end{equation}
For small enough $\varphi$, ${\eta}'\wedge (\diff{\eta}')^{n-1}$ is still a volume form on $S'=\{r'=1\}\cong S$, 
or in other words  ${\omega}'$ is still non-degenerate on the cone, and similarly ${\bar{g}'}$ will be a non-degenerate metric. Thus we have defined a new K\"ahler cone, and hence new Sasakian structure.

This deformation is precisely a deformation of the transverse K\"ahler metric, holding fixed its 
basic cohomology class in $H^{1,1}_B(\mathcal{F}_\xi)$. Indeed, clearly
\begin{equation}
\left(\omega'\right)^T = \omega^T + \frac{1}{2}\diff_B\diff_B^c \phi~.\nonumber
\end{equation}
In fact the converse is also true by the transverse $\partial\bar\partial$ lemma proven in \cite{ElK-A90}, which is thus 
an equivalent way to characterize these \emph{transverse K\"ahler deformations}. Notice that the basic first Chern class 
$c_1^B(\mathcal{F}_\xi)$ is invariant under such deformations, although the contact subbundle 
$D$ will change. \end{proof}

Later we shall also consider deforming the Reeb vector
field $\xi$ whilst holding the holomorphic structure of the cone fixed.

\subsection{Regularity}\label{sec:regularity}

There is a classification of Sasakian manifolds 
according to the global properties of the Reeb foliation $\mathcal{F}_\xi$. 
If the orbits of the Reeb vector field $\xi$ are all closed, 
and hence circles, then $\xi$ integrates to an isometric $U(1)$ action on $(S,g)$. 
Since $\xi$ is nowhere zero this action is locally free; that is, the 
isotropy group of every point in $S$ is finite. If the $U(1)$ action is 
in fact free then the Sasakian structure is said to be \emph{regular}. 
Otherwise, it is said to be \emph{quasi-regular}. If the orbits of 
$\xi$ are not all closed the Sasakian structure is said to be \emph{irregular}. 
In this case the closure of the 1-parameter subgroup of the isometry group 
of $(S,g)$ is isomorphic to a torus $\mathbb{T}^{k}$, for some positive integer 
$k$ called the \emph{rank} of the Sasakian structure. In particular, 
irregular Sasakian manifolds have at least a $\mathbb{T}^2$ isometry.

In the regular or quasi-regular case, the leaf space $Z=S/\mathcal{F}_\xi=S/U(1)$ has the structure of
a compact manifold or orbifold, respectively. In the latter case the orbifold 
singularities of $Z$ descend from the points in $S$ with non-trivial isotropy 
subgroups. Notice that, being finite subgroups of $U(1)$, these will all be isomorphic to cyclic groups. 
The transverse K\"ahler structure described above then pushes down to a 
K\"ahler structure on $Z$, so that $Z$ is a compact complex manifold or orbifold equipped with a K\"ahler 
metric $h$. 

\

\subsubsection*{Digression on orbifolds}

The reader will not need to know much about orbifolds in order to follow this article. However, 
we briefly digress here to sketch some basics, referring to \cite{BG99} or \cite{BG07} for a much more detailed account in the current context. 

Just as a manifold $M$ is a topological space that is locally modelled on $\R^k$, so an orbifold is 
a topological space locally modelled on $\R^k/\Gamma$, where $\Gamma$ is a finite group of diffeomorphisms. 
The local Euclidean charts $\{U_i,\varphi_i\}$ 
of a manifold are replaced with \emph{local uniformizing systems} $\{\tilde{U}_i,\Gamma_i,\varphi_i\}$. 
Here $\tilde{U}_i$ is an open subset of $\R^k$ containing the origin; 
$\Gamma_i$ is, without loss of generality, a finite subgroup of $O(k)$ acting effectively on $\R^k$; and $\varphi_i:\tilde{U}_i\rightarrow U_i$ is a continuous map onto the open set $U_i\subset M$ such that 
$\varphi_i\circ \gamma = \varphi_i$ for all $\gamma\in\Gamma_i$ and the induced map 
$\tilde{U}_i/\Gamma_i\rightarrow U_i$ is a homeomorphism. These charts are then glued together in an appropriate 
way. The groups $\Gamma_i$ are called the \emph{local 
uniformizing groups}. The least common multiple of the 
orders of the local uniformizing groups $\Gamma_i$, when it is defined, is called the \emph{order} of $M$ and denoted 
$\mathrm{ord}(M)$. In particular, $M$ is a manifold if and only if $\mathrm{ord}(M)=1$.
 One can similarly define complex orbifolds, where one may take the $\Gamma_i\subset U(k)$ 
acting on $\C^k$. 

If $x\in M$ is point and $p=\varphi^{-1}_i(x)$ then the conjugacy class of the isotropy subgroup $\Gamma_p\subset \Gamma_i$ depends only on $x$, not on the chart $\tilde{U}_i$. One denotes this $\Gamma_x$, so that the non-singular points of $M$ are 
those for which $\Gamma_x$ is trivial. The set of such points is dense in $M$. In the case at hand, 
where $M$ is realized as the leaf space $Z$ of a quasi-regular Reeb foliation, 
$\Gamma_x$ is the same as the leaf holonomy group of the leaf $x$. 

For orbifolds the notion of fibre bundle is modified to that of a fibre \emph{orbibundle}. 
These consist of bundles over the local uniformizing neighbourhoods $\tilde{U}_i$ that 
patch together in an appropriate way. In particular, part of the data specifying an orbibundle with 
structure group $G$ are group homomorphisms $h_i\in \mathrm{Hom}(\Gamma_i,G)$. 
The local uniformizing systems of an orbifold are glued together with the property that 
if $\phi_{ji}: \tilde{U}_i\rightarrow \tilde{U}_j$ is a diffeomorphism into its image 
then for each $\gamma_i\in \Gamma_i$ there is a unique $\gamma_j\in \Gamma_j$ such that 
$\phi_{ji}\circ \gamma_i = \gamma_j\circ \phi_{ji}$. 
The patching condition is then that if $B_i$ is a fibre bundle over $\tilde{U}_i$ there should exist a corresponding 
bundle map $\phi_{ij}^*:B_{j}\mid_{\phi_{ji}(\tilde{U}_i)}\rightarrow B_i$ such that 
$h_i(\gamma_i)\circ \phi_{ij}^*=\phi_{ij}^*\circ h_j(\gamma_j)$. Of course, by choosing an appropriate 
refinement of the cover we may assume that $B_i=\tilde{U}_i\times F$ where $F$ is the fibre on which $G$ acts. 
The total space is then itself an orbifold in which the $B_i$ may form the local uniformizing neighbourhoods. 
The group $\Gamma_i$ acts on $B_i$ by sending $(p_i,f)\in \tilde{U}_i\times F$ to 
$(\gamma^{-1}p_i,fh_i(\gamma))$, where $\gamma\in \Gamma_i$. Thus the local uniformizing 
groups of the total space may be taken to be subgroups of the $\Gamma_i$. In particular, when
$F=G$ is a Lie group so that we have a principal $G$ orbibundle, then the image $h_i(\Gamma_i)$
acts freely on the fibre. Thus provided the group homomorphisms $h_i$ \emph{inject} into the structure group $G$, 
the total space will in fact be a smooth manifold. This will be important in what follows.

The final orbinotion we need is that of \emph{orbifold cohomology}, introduced by 
Haefliger \cite{Hae84}. One may define the orbibundle $P$ of orthonormal frames over 
a Riemannian orbifold $(M,g)$ in the usual way. This is a principal $O(n)$ orbibundle, and 
the discussion in the previous paragraph implies that the total space $P$ is in fact a smooth manifold. 
One can then introduce the \emph{classifying space} $BM$ of the orbifold in an obvious way by 
defining $BM=(EO(n)\times P)/O(n)$, where $EO(n)$ denotes the universal 
$O(n)$ bundle and the action of $O(n)$ is diagonal. One then defines the orbifold 
homology, cohomology and homotopy groups as those of $BM$, respectively. In particular, 
the orbifold cohomology groups are denoted $H^*_{\mathrm{orb}}(M,\Z)=H^*(BM,\Z)$, and these reduce to the 
usual cohomology groups of $M$ when $M$ is a manifold. The projection $BM\rightarrow M$ 
has generic fibre the contractible space $EO(n)$, and this then induces an isomorphism 
$H^*_{\mathrm{orb}}(M,\R)\rightarrow H^*(M,\R)$. Typically integral classes 
 map to rational classes under the natural map $H^*_{\mathrm{orb}}(M,\Z)\rightarrow H^*_{\mathrm{orb}}(M,\R)
\stackrel{\sim}{\rightarrow} H^*(M,\R)$.

\

Returning to Sasakian geometry, in the regular or quasi-regular case the leaf space $Z=S/\mathcal{F}_\xi=S/U(1)$ 
is a manifold or orbifold, respectively. The Gysin sequence for the corresponding $U(1)$ (orbi)bundle 
then implies that the projection map $\pi:S\rightarrow Z$ gives rise to a ring isomorphism 
$\pi^*:H^*(Z,\R)\cong H^*_B(\mathcal{F}_\xi)$, thus relating the cohomology of the leaf space $Z$ to the 
basic cohomology of the foliation.

We may now state the following result \cite{BG00a}:
\begin{theorem}\label{thm:quotient}
 Let $\mathcal{S}$ be a compact regular or quasi-regular Sasakian manifold. Then the space of leaves of 
the Reeb foliation $\mathcal{F}_\xi$ is a compact K\"ahler manifold or orbifold $(Z,h,\omega_Z,J_Z)$, respectively. The corresponding projection
\begin{equation}
 \pi:(S,g)\rightarrow (Z,h)~,\nonumber 
\end{equation}
is a $($orbifold$)$ Riemannian submersion, with fibres being totally geodesic circles. 
Moreover, the cohomology class $[\omega_Z]$ is proportional to an integral class in 
 the $($orbifold$)$ cohomology group $H^2_{\mathrm{orb}}(Z,\mathbb{Z})$.
\end{theorem}
In either the regular or quasi-regular case, $\omega_Z$ is a closed 2-form on $Z$ which thus defines a cohomology class 
$[\omega_Z]\in H^2(Z,\R)$. In the regular case, the projection $\pi$ defines a principal $U(1)$ bundle, and 
$\omega_Z$ is proportional to the curvature 2-form of a unitary connection on this bundle. 
Thus $[\omega_Z]$ is proportional to a class in the image of the natural map 
$H^2(Z,\Z)\rightarrow H^2(Z,\R)$, since 
the curvature represents 
$2\pi c_1$ where $c_1$ denotes the first Chern class of the principal $U(1)$ bundle. 
In the quasi-regular case, the projection $\pi$ is instead a principal $U(1)$ orbibundle, with $\omega_Z$ 
again proportional to a curvature 2-form. The orbifold cohomology group 
$H^2_{\mathrm{orb}}(Z,\mathbb{Z})$  classifies isomorphism classes of principal $U(1)$ orbibundles over 
an orbifold $Z$, just as in the regular manifold case the first Chern class in $H^2(Z,\Z)$ classifies 
principal $U(1)$ bundles. The K\"ahler form $\omega_Z$ then defines a
cohomology class $[\omega_Z]\in H^2(Z,\R)$ 
which is proportional to a class in the image of the natural map $H^2_{\mathrm{orb}}(Z,\mathbb{Z})\rightarrow 
H^2_{\mathrm{orb}}(Z,\mathbb{R})\rightarrow H^2(Z,\mathbb{R})$. 

A K\"ahler manifold or orbifold whose K\"ahler class 
is proportional to an integral cohomology class in this way is called a \emph{Hodge} orbifold. 
There is no restriction on this constant of proportionality in Sasakian geometry: it may be changed via the $D$-homothetic
transformation defined in the next section.

The converse is also true \cite{BG00a}:
\begin{theorem}\label{thm:inverse}
 Let $(Z,h)$ be a compact Hodge orbifold. Let $\pi:S\rightarrow Z$ be a principal $U(1)$ orbibundle 
over $Z$ whose first Chern class is an integral class defined by $[\omega_Z]$, and let $\eta$ 
be a 1-form on $S$ with $\diff\eta=2\pi^*\omega_Z$ ($\eta$ is then proportional to a connection 1-form). 
Then $(S,\pi^*h+\eta\otimes\eta)$ is a Sasakian orbifold. Furthermore, if all the 
local uniformizing groups inject into the structure group $U(1)$ $($the $h_i\in \mathrm{Hom}(\Gamma_i,U(1))$ are all injective$)$, then the total space $S$ is a smooth manifold.
\end{theorem}

We close this subsection by noting that
\begin{equation}
i_\xi \omega = -\frac{1}{2}\diff r^2~,
\end{equation}
where recall that $\omega$ is the K\"ahler form on the cone $(C(S),\bar{g})$. Thus $\tfrac{1}{2}r^2$ is precisely the \emph{Hamiltonian function}  for the Reeb vector field $\xi$. 
In the regular/quasi-regular case the K\"ahler manifold/orbifold $(Z,h,\omega_Z,J_Z)$ may then be viewed 
as the \emph{K\"ahler reduction} of the K\"ahler cone with respect to the
corresponding Reeb $U(1)$ action.

\subsection{The Einstein condition}\label{sec:Einstein}

We begin with the following more general definition:
\begin{definition}
A Sasakian manifold $\mathcal{S}=(S,g,\eta,\xi,\Phi)$ is said to be \emph{$\eta$-Sasaki-Einstein} if there are constants $\lambda$ and $\nu$ such that
\begin{equation}
\mathrm{Ric}_g = \lambda g + \nu \eta\otimes\eta~.\nonumber
\end{equation}
\end{definition}
An important fact is that $\lambda+\nu=2(n-1)$. This follows from the second condition in Proposition 
\ref{prop:equiv}, which implies that for a Sasakian manifold $\mathrm{Ric}_g(\xi,\xi)=2(n-1)$. In particular, Sasaki-Einstein manifolds, with $\nu=0$, necessarily have $\lambda=2(n-1)$. 
\begin{definition}
 A Sasaki-Einstein manifold is a Sasakian manifold $(S,g)$ with $\mathrm{Ric}_g=2(n-1)g$.
\end{definition}

It is easy to see that the $\eta$-Sasaki-Einstein condition is equivalent to the transverse 
K\"ahler metric being Einstein, so that $\mathrm{Ric}^T = \kappa g^T$ for some constant $\kappa$. To see the equivalence one notes that
\begin{equation}\label{riccis}
\mathrm{Ric}_g(\tilde{X},\tilde{Y}) = \mathrm{Ric}^T(X,Y) - 2g^T(X,Y)~,
\end{equation}
where $X$, $Y$ are vector fields on the local leaf spaces $\{V_\alpha\}$ and $\tilde{X}$, $\tilde{Y}$ are lifts to $D$.
Then $\mathrm{Ric}^T = \kappa g^T$ together with (\ref{riccis}) implies 
that
\begin{equation}
\mathrm{Ric}_g = (\kappa -2)g  + (2n-\kappa)\eta\otimes\eta~.\nonumber
\end{equation}

Given a Sasakian manifold $\mathcal{S}$ one can check that for a constant $a>0$ the rescaling
\begin{equation}\label{scaling}
 g' = ag + (a^2-a)\eta\otimes\eta~, \quad \eta'=a\eta~,\quad \xi' = \frac{1}{a}\xi~, \quad \Phi'=\Phi~,
\end{equation}
gives a Sasakian manifold $(S,g',\eta',\xi',\Phi')$ with the same holomorphic structure on $C(S)$, but with 
$r'=r^a$. This is known as a \emph{$D$-homothetic transformation} \cite{Tan68}. Using the above formulae, 
together with the fact that the Ricci tensor is invariant under scaling the metric by 
a positive constant,  it is then straightforward to show that if $\mathcal{S}=(S,g,\eta,\xi,\Phi)$ is $\eta$-Sasaki-Einstein 
with constant $\lambda>-2$, then a $D$-homothetic transformation with $a=(\lambda+2)/2n$ gives 
a Sasaki-Einstein manifold. Thus any Sasakian structure which is transversely K\"ahler-Einstein with 
$\kappa>0$ may be transformed via this scaling to a Sasaki-Einstein structure.

The Sasaki-Einstein case may be summarized by the following:
\begin{proposition}\label{prop:SE}
Let $(S,g)$ be a Sasakian manifold of dimension $2n-1$. Then the following are equivalent
\begin{enumerate}
\item $(S,g)$ is Sasaki-Einstein with $\mathrm{Ric}_g = 2(n-1)g$.
\item The K\"ahler cone $(C(S),\bar{g})$ is Ricci-flat, $\mathrm{Ric}_{\bar{g}}=0$.
\item The transverse K\"ahler structure to the Reeb foliation $\mathcal{F}_\xi$ is 
K\"ahler-Einstein with $\mathrm{Ric}^T = 2n g^T$.
\end{enumerate}
\end{proposition}
It immediately follows that the restricted holonomy group $\mathrm{Hol}^0(\bar{g})\subset SU(n)$.  
Notice that a Sasaki-Einstein 3-manifold has a universal covering space which is isometric to the standard 
round sphere, so the first interesting dimension is $n=3$, or equivalently real dimension $\dim S=5$.

Since $\rho^T$ represents $2\pi c_1^B(\mathcal{F}_\xi)\in H^{1,1}_B(\mathcal{F}_\xi)$, clearly 
a necessary condition for a Sasakian manifold to admit a transverse K\"ahler deformation to a 
Sasaki-Einstein structure, in the sense of Proposition \ref{prop:def}, is that $c_1^B=c_1^B(\mathcal{F}_\xi)>0$. 
Indeed, we have the following result, formalized in \cite{FOW09}:
\begin{proposition}\label{prop:equivGor}
The following necessary conditions for a Sasakian manifold $\mathcal{S}$ to admit a 
deformation of the transverse K\"ahler structure to a Sasaki-Einstein metric are equivalent:
\begin{enumerate}
\item $c_1^B=a[\diff\eta]\in H^{1,1}_B(\mathcal{F}_\xi)$ for some positive constant $a$.
\item $c_1^B>0$ and $c_1(D)=0\in H^2(S,\R)$.
\item For some positive integer $\ell>0$, the $\ell$th power of the canonical line bundle $K_{C(S)}^\ell$ admits 
a nowhere vanishing holomorphic section $\Omega$ with $\mathcal{L}_\xi \Omega= \ii n \Omega$.
\end{enumerate}
\end{proposition}
As described in \cite{vanC09}, the space $X=C(S)\cup \{r=0\}$, obtained by adding the cone point at 
$\{r=0\}$ to $C(S)\cong \R_{>0}\times S$, can be made into a complex analytic space in a unique way. 
In fact it is simple to see that $X$ is Stein, and the point $o=\{r=0\}\in X$ is an isolated singularity. 
Then (3) above implies that, by definition, $X$ is $\ell$-Gorenstein:
\begin{definition}\label{def:Gor}
An analytic space $X$ with an isolated singularity $o\in X$ is said to be $\ell$-\emph{Gorenstein} if 
$K_{X\setminus \{o\}}^\ell$ is trivial. In particular, if $\ell=1$ one says that $X$ is \emph{Gorenstein}.
\end{definition}
\begin{proof} (Proposition \ref{prop:equivGor}) The equivalence of (1) and (2) follows immediately from the long 
exact sequence \cite{Ton97} relating the basic cohomology of the foliation $\mathcal{F}_\xi$ to the cohomology of $S$ (see \cite{FOW09}). 
The Ricci form $\rho$ of the cone $\left(C(S),\bar{g}\right)$ is 
related to the transverse Ricci form, by an elementary calculation, via $\rho=\rho^T-n\diff\eta$. Here we are regarding $\rho^T$ as a global basic 2-form on 
$S$, pulled back to the cone $C(S)$. If condition (1) holds then by the above comments there is a $D$-homothetic transformation 
so that $[\rho^T]=n[\diff\eta]\in H^{1,1}_B(\mathcal{F}_\xi)$. It now follows from the transverse $\partial\bar\partial$ lemma 
\cite{ElK-A90} that there is a smooth function $f$ on $C(S)$ with $r\partial_r f= \xi f=0$ and $\rho=\ii \partial\bar\partial f$ ($f$ is the pull-back of a basic function on $S$).
But now $\mathrm{e}^f\omega^n/n!$ defines a flat metric on $K_{C(S)}$, where recall that $\omega$ is the K\"ahler form for $\bar{g}$. 
It follows that there is a holomorphic section $\Omega$ of $K^\ell_{C(S)}$, for some positive integer $\ell>0$, with $\|\Omega\|=1$. 
Using the fact that $f$ is invariant under $r\partial_r$ and that $\omega$ is homogeneous degree 2, the equality of volume forms
\begin{equation}\nonumber
\frac{\ii^n}{2^n}(-1)^{n(n-1)/2} \Omega\wedge\bar{\Omega} = \mathrm{e}^f\frac{\omega^n}{n!}~,
\end{equation}
implies that $\mathcal{L}_{r\partial_r}\Omega = n\Omega$. \end{proof}

\subsection{3-Sasakian manifolds}\label{sec:3sas}

In dimensions of the form $n=2p$, so $\dim S = 4p-1$, there exists a special class of Sasaki-Einstein 
manifolds called \emph{3-Sasakian} manifolds:
\begin{definition}
A Riemannian manifold $(S,g)$ is \emph{3-Sasakian} if and only if its metric cone
$\left(C(S)=\R_{>0}\times S, \bar{g}=\diff r^2 + r^2 g\right)$ is hyperK\"ahler. 
\end{definition}
This implies that the cone has complex dimension $n=2p$, or real dimension $4p$, and that the 
holonomy group $\mathrm{Hol}(\bar{g})\subset Sp(p)\subset SU(2p)$. Thus 3-Sasakian manifolds 
are automatically Sasaki-Einstein. The hyperK\"ahler structure on the cone 
descends to a 3-Sasakian structure on the base of the cone $(S,g)$. In particular, the triplet 
of complex structures gives rise to a triplet of Reeb vector fields $(\xi_1,\xi_2,\xi_3)$ whose 
Lie brackets give a copy of the Lie algebra $\mathtt{su}(2)$. There is then a corresponding 
3-dimensional foliation, whose leaf space is a \emph{quaternionic K\"ahler} manifold or orbifold.
This extra structure means that 
3-Sasakian geometry is rather more constrained, and it is somewhat more straightforward to 
construct examples. 
In particular, rich infinite classes of examples were produced in the 1990s 
via a quotient construction (essentially 
the hyperK\"ahler quotient). 
A review of those developments was given in a previous article in this journal series \cite{BG99},  with a more 
recent account appearing in \cite{BG05}. 
We note that the first non-trivial dimension for a 3-Sasakian manifold is $\dim S=7$, 
and also that 3-Sasakian manifolds are automatically regular or quasi-regular as Sasaki-Einstein manifolds 
(indeed, the first quasi-regular Sasaki-Einstein manifolds constructed were 3-Sasakian 7-manifolds). 
We will therefore not discuss 3-Sasakian geometry any further 
in this article, but focus instead on the construction of Sasaki-Einstein manifolds that are not 3-Sasakian. 

\subsection{Killing spinors}\label{sec:spinors}

For applications to supergravity theories one wants a slightly stronger definition of 
Sasaki-Einstein manifold than we have given above. This is related to the following:
\begin{definition}
 Let $(S,g)$ be a complete Riemannian spin manifold. Denote the spin bundle by $\mathscr{S}S$ and let $\psi$ be a 
smooth section of $\mathscr{S}S$. Then $\psi$ is said to be a \emph{Killing spinor} if for some constant $\alpha$
\begin{equation}
 \nabla_Y \psi = \alpha Y\cdot \psi~,
\end{equation}
for every vector field $Y$, where $\nabla$ denotes the spin connection of $g$ and $Y\cdot\psi$ is Clifford multiplication 
of $Y$ on $\psi$. One says that $\psi$ is \emph{imaginary} if $\alpha\in \mathrm{Im}(\C^*)$, \emph{parallel} if $\alpha=0$, 
or \emph{real} if $\alpha\in \mathrm{Re}(\C^*)$.
\end{definition}
It is a simple exercise to show that the existence of such a Killing spinor implies 
that $g$ is Einstein with constant $\lambda=4(m-1)\alpha^2$, where $m=\dim S$. In particular, the existence of a real 
Killing spinor implies that $(S,g)$ is a compact Einstein manifold with positive Ricci curvature. The relation to Sasaki-Einstein 
geometry is given by the following result of \cite{Bar93}:
\begin{theorem}\label{thm:spinors}
 A complete simply-connected Sasaki-Einstein manifold admits at least 2 linearly independent real Killing spinors 
with $\alpha=+\tfrac{1}{2},-\tfrac{1}{2}$ for $n=2p-1$ and $\alpha=+\tfrac{1}{2},+\tfrac{1}{2}$ for $n=2p$, respectively. 
Conversely, a complete Riemannian spin manifold admitting such Killing spinors in these dimensions is Sasaki-Einstein 
with $\mathrm{Hol}(\bar{g})\subset SU(n)$.
\end{theorem}
Notice that in both cases $\mathrm{Hol}(\bar{g})\subset SU(n)$, so that in particular a simply-connected Sasaki-Einstein manifold is indeed spin. 
Moreover, in this case $\ell=1$ in Proposition \ref{prop:equivGor} so that the singularity $X=C(S)\cup\{r=0\}$ is Gorenstein. 
Indeed, a real Killing spinor on $(S,g)$ lifts to a \emph{parallel} spinor on $(C(S),\bar{g})$ \cite{Bar93}, 
and from this parallel spinor one can construct a nowhere zero holomorphic $(n,0)$-form by ``squaring'' it. 
We refer the reader to \cite{MSY07} for details.

When a Sasaki-Einstein manifold is not simply-connected the existence of Killing spinors is more subtle. An instructive example is $S^5$, equipped 
with its standard metric. Here $X=\C^3$ is equipped with its flat K\"ahler metric. Denoting standard complex coordinates on $\C^3$ by
$(z_1,z_2,z_3)$  we may consider the quotient $S^5/\Z_q$, where $\Z_q$ acts by sending $(z_1,z_2,z_3)\mapsto (\zeta z_1,\zeta z_2,\zeta z_3)$ 
with $\zeta$ a primitive $q$th root of unity. 
For $q=2$ this is the antipodal map, giving $\mathbb{RP}^5$ which is not even a spin manifold. In fact of all these quotients 
only $S^5$ and $S^5/\Z_3$ admit Killing spinors. 

For applications to supergravity theories, a Sasaki-Einstein manifold is in fact 
defined to satisfy this stronger requirement that it admits Killing spinors.
Of course, since $\pi_1(S)$ is finite by Myers' Theorem \cite{M41}, one may always lift to a simply-connected cover, where Theorem \ref{thm:spinors} implies that the two definitions 
coincide. We shall thus not generally emphasize this distinction.

The reader might wonder what happens to Theorem \ref{thm:spinors} when the number of linearly independent Killing spinors is not 2. For simplicity we focus on the simply-connected case.
 When $n=2p-1$, the existence of 1 Killing spinor in fact implies the existence of 2 with opposite sign of $\alpha$, so that $(S,g)$ is Sasaki-Einstein. If there are more 
than 2, or at least 2 with the same sign of $\alpha$, then $(S,g)$ is the round sphere. When $n=4$, so that $\dim S=7$, 
it is possible for a Riemannian spin 7-manifold $(S,g)$ to admit a single real Killing spinor, in which case 
$(S,g)$ is said to be a \emph{weak $G_2$ holonomy manifold}; the metric cone then has holonomy contained in the group $Spin(7)\subset SO(8)$. In all other dimensions of the form 
$n=2p$, the existence of 1 Killing spinor again implies the existence of 2, implying $(S,g)$ is Sasaki-Einstein. A simply-connected 3-Sasakian manifold has 
3 linearly independent Killing spinors, all with $\alpha=+\tfrac{1}{2}$. If there are more than 3, or at least 2 
with opposite sign of $\alpha$, then again 
$(S,g)$ is necessarily the round sphere. For further details, and a list of references, the reader is referred to \cite{BG99}.

\section{Regular Sasaki-Einstein manifolds}\label{sec:regular}

\subsection{Fano K\"ahler-Einstein manifolds}

Theorem \ref{thm:quotient}, together with Proposition \ref{prop:SE}, implies that any regular Sasaki-Einstein 
manifold is 
the total space of a principal $U(1)$ bundle over a K\"ahler-Einstein manifold $(Z,h)$.
 On the other hand, Theorem \ref{thm:inverse} implies 
that the converse is also true. In fact this construction of Einstein metrics on the total spaces of 
principal $U(1)$ bundles over K\"ahler-Einstein manifolds is in the very early paper of Kobayashi 
\cite{Kob56}. 
\begin{theorem}\label{thm:regular}
A complete regular Sasaki-Einstein manifold $(S,g)$ of dimension $(2n-1)$ is the total space of a 
principal $U(1)$ bundle over a compact K\"ahler-Einstein manifold $(Z,h,\omega_Z)$ with
positive Ricci curvature $\mathrm{Ric}_h=2n h$, which is the leaf space of the Reeb foliation $\mathcal{F}_\xi$. 
If $S$ is simply-connected then this $U(1)$ bundle has first Chern class $-[n\omega_Z/\pi I(Z)]=-c_1(Z)/I(Z)$, 
where $I(Z)\in \Z_{>0}$ is the \emph{Fano index} of $Z$. 

Conversely, if $(Z,h,\omega_Z)$ is a complete simply-connected  K\"ahler-Einstein manifold with
positive Ricci curvature $\mathrm{Ric}_h=2n h$, then let $\pi:S\rightarrow Z$ be the principal $U(1)$ bundle 
with first Chern class $-c_1(Z)/I(Z)$. Then $g=\pi^*h + \eta\otimes\eta$ is a regular Sasaki-Einstein 
metric on the simply-connected manifold $S$, where $\eta$ is the connection 1-form on $S$ with 
curvature $\diff\eta = 2\pi^*\omega_Z$.
\end{theorem}
Recall here:
\begin{definition}
A \emph{Fano manifold} is a compact complex manifold $Z$ with positive first Chern class 
$c_1(Z)>0$. The \emph{Fano index} $I(Z)$ is the largest positive integer such that 
$c_1(Z)/I(Z)$ is an integral class in the group of holomorphic line bundles $\mathrm{Pic}(Z)=H^2(Z,\Z)\cap H^{1,1}(Z,\R)$. 
\end{definition}
In particular, K\"ahler-Einstein manifolds with positive Ricci curvature are Fano. Notice that the principal 
$U(1)$ bundle in Theorem \ref{thm:regular} is that associated to the line bundle $K_Z^{1/I(Z)}$, where 
$K_Z$ is the canonical line bundle of $Z$. Also notice 
 that by taking a $\Z_m\subset U(1)$ quotient 
of a simply-connected $S$ in Theorem \ref{thm:regular}, where $U(1)$ acts via the free Reeb action, we also obtain 
a regular Sasaki-Einstein manifold with $\pi_1(S/\Z_m)\cong \Z_m$; this is equivalent to taking the $m$th power 
of the principal $U(1)$ bundle, which has associated line bundle $K_Z^{m/I(Z)}$. However, the Killing spinors on $(S,g)$ guaranteed by Theorem \ref{thm:spinors} 
are invariant under $\Z_m$ only when $m$ divides the Fano index $I(Z)$. Only in these cases is the quotient 
Sasaki-Einstein in the stronger sense of admitting a real Killing spinor.

Via Theorem \ref{thm:regular} the classification of regular Sasaki-Einstein manifolds effectively reduces to classifying Fano K\"ahler-Einstein manifolds. 
This is a rich and deep subject, which is still very much an active area of research. Below we give a brief overview of some key results.

\subsection{Homogeneous Sasaki-Einstein manifolds}

\begin{definition}
 A Sasakian manifold $\mathcal{S}$ is said to be \emph{homogeneous} if there is a transitively acting group $G$ of isometries 
preserving the Sasakian structure.
\end{definition}
If $S$ is compact, then $G$ is necessarily a compact Lie group. We then have the following theorem of \cite{BG00a}:
\begin{theorem}\label{thm:homo}
 Let $(S,g')$ be a complete homogeneous Sasakian manifold with $\mathrm{Ric}_{g'}\geq \epsilon > -2$. 
Then $(S,g')$ is a compact regular homogeneous Sasakian manifold, and there is a homogeneous Sasaki-Einstein 
metric $g$ on $S$ that is compatible with the underlying contact structure. Moreover, $S$ is the total space 
of a principal $U(1)$ bundle over a generalized flag manifold $K/P$, equipped with its K\"ahler-Einstein metric. 
Via Theorem \ref{thm:regular}, the converse is also true.
\end{theorem}
Recall here that a \emph{generalized flag manifold} $K/P$ is a homogeneous space where $K$ is a complex 
semi-simple Lie group, and $P$ is any complex subgroup of $K$ that contains a Borel subgroup (so that $P$ is a \emph{parabolic} subgroup of $K$). 
It is well-known that $K/P$ is Fano and admits a homogeneous K\"ahler-Einstein metric \cite{Bes87}. Conversely, 
any compact homogeneous simply-connected K\"ahler-Einstein manifold is a generalized flag manifold. The metric on 
$K/P$ is $G$-invariant, where $G$ is a maximal compact subgroup of $K$, and one can write $K/P=G/H$ 
for appropriate subgroup $H$.

In low dimensions Theorem \ref{thm:homo} leads \cite{BG99} to the following list,  well-known to supergravity theorists:
\begin{corollary}\label{cor:homo}
 Let $(S,g)$ be a complete homogeneous Sasaki-Einstein manifold of dimension $2n-1$. Then $S$ is a 
principal $U(1)$ bundle over
\begin{enumerate}
 \item $\mathbb{CP}^1$ when $n=2$,
 \item $\mathbb{CP}^2$ or $\mathbb{CP}^1\times\mathbb{CP}^1$ when $n=3$,
 \item $\mathbb{CP}^3$, $\mathbb{CP}^2\times \mathbb{CP}^1$, $\mathbb{CP}^1\times\mathbb{CP}^1\times\mathbb{CP}^1$, 
$SU(3)/\mathbb{T}^2$, or the real Grassmannian $\mathrm{Gr}_2(\mathbb{R}^5)$ of 2-planes in $\R^5$ when $n=4$.
\end{enumerate}
\end{corollary}

\subsection{Regular Sasaki-Einstein 5-manifolds}

As mentioned in the introduction, regular Sasaki-Einstein 5-manifolds are classified \cite{FK89, BFGK91}. This is thanks to the classification of 
Fano K\"ahler-Einstein surfaces due to Tian-Yau \cite{T87, T90, TY87}.
\begin{theorem}\label{thm:regular5}
 Let $(S,g)$ be a regular Sasaki-Einstein 5-manifold. Then $S=\tilde{S}/\Z_m$, where the universal cover $(\tilde{S},g)$ is one of the following:
\begin{enumerate}
 \item $S^5$ equipped with its standard round metric. Here $Z=\mathbb{CP}^2$ equipped with its 
standard Fubini-Study metric.
 \item The Stiefel manifold $V_2(\R^4)\cong S^2\times S^3$ of 2-frames in $\R^4$. Here 
 $Z=\mathbb{CP}^1\times\mathbb{CP}^1$ equipped with the symmetric product of round metrics on each $\mathbb{CP}^1\cong S^2$.
 \item The total space $S_k$ of the principal $U(1)$ bundles $S_k\rightarrow P_k$, for $3\leq k\leq 8$, where 
$P_k = \mathbb{CP}^2\# k\overline{\mathbb{CP}^2}$ is the $k$-point blow-up of $\mathbb{CP}^2$. For each complex structure 
on these del Pezzo surfaces there is a unique K\"ahler-Einstein metric, up to automorphism \cite{Siu88, T87,  T90, TY87}, and 
a corresponding unique Sasaki-Einstein metric $g$ on $S_k\cong \#k \left(S^2\times S^3\right)$. 
In particular, for $5\leq k\leq 8$ by varying the complex structure this gives a complex $2(k-4)$-dimensional family of regular Sasaki-Einstein structures. 
\end{enumerate}
\end{theorem}
Cases (1) and (2) are of course the 2 homogeneous spaces listed in (2) of Corollary \ref{cor:homo},
 and so the metrics are easily written down explicitly. 
The Sasaki-Einstein metric in case (2) was first noted by Tanno in \cite{Tan79}, although in the physics literature 
the result is often attributed to Romans \cite{Rom85}. In the latter case the manifold is referred to as $T^{11}$, 
the $T^{pq}$ being homogeneous Einstein metrics on principal $U(1)$ bundles over 
$\mathbb{CP}^1\times\mathbb{CP}^1$ with Chern numbers $(p,q)\in \Z\oplus\Z\cong H^2(\mathbb{CP}^1\times\mathbb{CP}^1,\Z)$. This is a generalization of the Kobayashi construction 
\cite{Kob56}, and was further generalized to torus bundles  by Wang-Ziller in \cite{WZ90}.
The corresponding Ricci-flat K\"ahler cone over $T^{11}$ has the complex structure of the quadric 
singularity $\{z_1^2+z_2^2+z_3^2+z_4^2=0\}\subset\C^4$ minus the isolated singular point at the origin. 
This hypersurface singularity is called the ``conifold'' in the string theory literature, and there are literally hundreds
of papers that study different aspects of string theory on this space. 

The K\"ahler-Einstein metrics on del Pezzo surfaces in (3) are known to exist, but are not known in explicit form. The complex structure 
moduli simply correspond to moving the blown-up points. The fact that $S_k$ is diffeomorphic to $\#k \left(S^2\times S^3\right)$ 
follows from Smale's Theorem \cite{Sm75}:
\begin{theorem}\label{thm:smale}
 A compact simply-connected spin 5-manifold $S$ with no torsion in $H_2(S,\Z)$ is diffeomorphic to $\#k \left(S^2\times S^3\right)$. 
\end{theorem}
The homotopy and homology groups of $S_k$ are of course straightforward to compute using 
their description as principal $U(1)$ bundles over $P_k$.

By Theorem \ref{thm:spinors}, in each case the simply-connected cover $(\tilde{S},g)$ admits 2 real Killing spinors. 
Only for $m=3$ in case (1) and $m=2$ in case (2) do the $\Z_m$ quotients also admit such Killing spinors 
\cite{FK89}.

\subsection{Existence of K\"ahler-Einstein metrics}\label{sec:existence}

For a K\"ahler manifold $(Z,h,\omega,J)$  the Einstein equation $\mathrm{Ric}_h = \kappa h$ is of course 
equivalent to the 2-form equation $\rho_h=\kappa \omega_h$, where $\rho_h=-\ii \partial \bar\partial \log \det h$ denotes the Ricci form of the metric $h$. Since the cohomology class $[\rho]\in H^{1,1}(Z,\R)$ of the Ricci form equals 
$2\pi c_1(Z)$, it follows that on such a K\"ahler-Einstein manifold $2\pi c_1(Z)=\kappa [\omega]$. Notice that for $\kappa=0$ $Z$ is Calabi-Yau 
($c_1(Z)=0$) and there is no restriction on the K\"ahler class, while 
for $\kappa>0$ or $\kappa<0$ instead $Z$ must be either Fano or anti-Fano ($c_1(Z)<0$), respectively, 
and in either case here the K\"ahler class is fixed uniquely. 

Suppose that $\omega=\omega_h$ is a K\"ahler 2-form on $Z$ with $\kappa [\omega] = 2\pi c_1(Z)\in H^{1,1}(Z,\R)$. 
By the $\partial\bar\partial$ lemma, there exists a global real function $f\in C^\infty(Z)$ such that
\begin{equation}\label{disc}
\rho_h - \kappa \omega_h = \ii \partial\bar\partial f~.
\end{equation}
The function $f$ is often called the \emph{discrepancy potential}. It is unique up to an additive constant, 
and the latter may be conveniently fixed by requiring, for example, $\int_Z \left(\mathrm{e}^f-1\right)\omega_h^{n-1}=0$, where $\dim_\C Z=n-1$. Notice 
that $f$ is essentially the same function $f$ appearing in the proof of 
Proposition \ref{prop:equivGor}. (More precisely, the function there is the pull-back 
of $f$ here under the $\C^*=\R_{>0}\times U(1)$ quotient $C(S)\rightarrow Z$ for a 
regular Sasakian structure with leaf space $Z$.)

On the other hand, if $g$ is a K\"ahler-Einstein metric, with $[\omega_{g}]=[\omega_h]$ 
and $\rho_{g}=\kappa \omega_{g}$, then the $\partial\bar\partial$ lemma 
again gives a real function $\phi\in C^\infty(Z)$ such that 
\begin{equation}\label{Kahlerdiff}
\omega_{g}-\omega_{h} = \ii \partial\bar\partial\phi~.
\end{equation}
Thus $\rho_h - \rho_g = \ii\partial\bar\partial (f-\kappa\phi)$, or relating the volume 
forms as $\omega_g^{n-1}=\mathrm{e}^F \omega_{h}^{n-1}$ with $F\in C^\infty(Z)$ 
equivalently
\begin{equation}\nonumber
\ii\partial\bar\partial F = \ii \partial\bar\partial (f-\kappa\phi)~.
\end{equation}
This implies $F=f-\kappa\phi+c$ with $c$ a constant. Again, this may be fixed 
by requiring, for example, 
\begin{equation}\label{constfix}
\int_Z \left(\mathrm{e}^{f-\kappa\phi}-1\right)\omega_h^{n-1}=0~.
\end{equation} 
We have then shown the following:
\begin{proposition}\label{prop:MA}
Let $(Z,J)$ be a compact K\"ahler manifold, of dimension $\dim_\C Z = n-1$, with K\"ahler metrics $h$, $g$ in the same K\"ahler 
class, $[\omega_h]=[\omega_g]\in H^{1,1}(Z,\R)$, and with $\kappa [\omega_h]=2\pi c_1(Z)$. 
Let $f,\phi\in C^\infty(Z)$ be the functions defined via (\ref{disc}) and (\ref{Kahlerdiff}), 
and with the relative constant of $f-\kappa\phi$ fixed by (\ref{constfix}). Then the metric $g$ is 
K\"ahler-Einstein with constant $\kappa$ if and only if $\phi$ satisfies the Monge-Amp\`ere equation
\begin{equation}\nonumber
\omega_g^{n-1} = \mathrm{e}^{f-\kappa\phi} \omega_h^{n-1}~,
\end{equation}
or equivalently
\begin{equation}\label{MA}
\frac{\det \left(h_{i\bar{j}}+\frac{\partial^2\phi}{\partial z_i\partial \bar{z}_j}\right)}{\det h_{i\bar{j}}} = 
\mathrm{e}^{f-\kappa\phi}~,
\end{equation}
where $z_1,\ldots,z_{n-1}$ are local complex coordinates on $Z$.
\end{proposition}

For $\kappa<0$ this problem was solved independently by Aubin \cite{Aub76} and Yau \cite{Yau78}. 
Without loss of generality, we may rescale the metric so that $\kappa=-1$ and then state:
\begin{theorem}
Let $(Z,J)$ be a compact K\"ahler manifold with $c_1(Z)<0$. Then there exists a unique 
K\"ahler-Einstein metric with $\rho_g = -\omega_g$.
\end{theorem}
The proof relies on the Maximum Principle. The Calabi-Yau case $\kappa=0$ is substantially harder, 
and was proven in Yau's celebrated paper:
\begin{theorem}
Let $(Z,J)$ be a compact K\"ahler manifold with $c_1(Z)=0$. Then there exists a unique Ricci-flat K\"ahler 
metric in each K\"ahler class.
\end{theorem}

On the other hand, the problem for $\kappa>0$ is still open. In particular, there are known obstructions to solving 
the Monge-Amp\`ere equation (\ref{MA}) in this case. On the other hand, it is known that if there is a solution, it 
is unique up to automorphism \cite{BM}. In the remainder of this section we give a very brief overview of the Fano 
$\kappa>0$ case, referring the reader to the literature for further details.

In \cite{Mat57} Matsushima proved that for a Fano K\"ahler-Einstein manifold the complex
Lie algebra $\mathfrak{a}(Z)$ of holomorphic vector fields is the complexification of the Lie algebra of 
\emph{Killing} vector fields. Since the isometry group of a compact Riemannian manifold is a 
compact Lie group, in particular this implies that $\mathfrak{a}(Z)$
is necessarily \emph{reductive}; that is, $\mathfrak{a}(Z)=\mathcal{Z}(\mathfrak{a}(Z))\oplus \left[
\mathfrak{a}(Z),\mathfrak{a}(Z)\right]$, where $\mathcal{Z}(\mathfrak{a}(Z))$ denotes the centre. 
The simplest such obstructed examples are in fact the 1-point and 2-point blow-ups of $\mathbb{CP}^2$ that are not listed 
in Theorem \ref{thm:regular5}, despite these being Fano surfaces. Matsushima's result implies that the isometry group 
of a Fano K\"ahler-Einstein manifold is a maximal compact subgroup of the automorphism group.

Another obstruction, also related to holomorphic vector fields on $Z$, is the Futaki invariant of \cite{Fut83}. 
If $\zeta\in \mathfrak{a}(Z)$ is a holomorphic vector field then define
\begin{equation}\label{Fut}
\mathcal{F}(\zeta) = \int_Z \zeta(f)\, \omega_h^{n-1},
\end{equation}
where $f$ is the discrepancy potential defined via (\ref{disc}). The function $\mathcal{F}$ is independent 
of the choice of K\"ahler metric $h$ in the K\"ahler class $[\omega_h]$, and defines a 
Lie algebra homomorphism $\mathcal{F}:\mathfrak{a}(Z)\rightarrow\C$. For this reason it 
is also sometimes called the \emph{Futaki character}. Moreover, Mabuchi \cite{Mab90} 
proved that the nilpotent radical of $\mathfrak{a}(Z)$ lies in $\ker \mathcal{F}$, so that 
$\mathcal{F}$ is completely determined by its restriction to the maximal reductive subalgebra. 
Since $\mathcal{F}$ is a Lie algebra character, it also vanishes on the derived algebra 
$\left[\mathfrak{a}(Z),\mathfrak{a}(Z)\right]$, and therefore the Futaki invariant is determined entirely by 
its restriction to the centre of $\mathfrak{a}(Z)$. In practice, $\mathcal{F}$ may be computed 
via localization; see, for example, the formula in \cite{Tian99}. 
Clearly $f$ is constant for a K\"ahler-Einstein metric, 
and thus the Futaki invariant must vanish in this case. Indeed, both the 1-point and 2-point blow-ups 
of $\mathbb{CP}^2$ also have non-zero Futaki invariants, and are thus obstructed this way also. 
The Futaki invariant will turn out  
to be closely related to a natural construction in Sasakian geometry, that generalizes to the quasi-regular 
and irregular cases.

Since both obstructions above are related to the Lie algebra of holomorphic vector fields on the Fano $Z$, 
there was a conjecture that in the absence of holomorphic vector fields there would be no obstruction 
to the existence of a K\"ahler-Einstein metric. However, a counterexample was later given 
by Tian in \cite{Ti3}. 

In fact it is currently believed that a Fano manifold $Z$ admits a K\"ahler-Einstein metric if and only if 
it is \emph{stable}, in an appropriate geometric invariant theory sense. This idea goes back to 
Yau \cite{Yau93}, and has been developed by Donaldson, Tian, and others. It is clearly beyond the scope 
of this article to describe this still very active area of research. However, the basic idea is to 
use $K_Z^{-k}$ for $k\gg0$ to embed $Z$ into a large 
complex projective space $\mathbb{CP}^{N_k}$ via its space of sections (the Kodaira embedding). Then stability 
of $Z$ is taken in the geometric invariant theory sense, for the automorphisms of these 
projective spaces, as $k\rightarrow\infty$. A stable orbit should contain a zero of a corresponding moment map, 
and in the present case this amounts to saying that by acting with an appropriate automorphism of 
$\mathbb{CP}^{N_k}$ a \emph{stable} $Z$ can be moved to a \emph{balanced embedding}, in the sense of Donaldson \cite{Don01}. 
For a sequence of balanced embeddings, the pull-back of the Fubini-Study metric on $\mathbb{CP}^{N_k}$ as $k\rightarrow\infty$ 
should then approach the K\"ahler-Einstein metric on $Z$. The precise notion of stability here is called 
\emph{K-stability}. 

In practice, even if the above stability conjecture was settled it is difficult to check in practice for a 
given Fano manifold. More practically, 
one can sometimes prove existence of solutions to 
(\ref{MA}) in appropriate examples using the continuity method. Thus, for appropriate classes of examples, 
one can often write down sufficient conditions for a solution, although these conditions are in general not 
expected to be necessary. This will be the pragmatic approach followed in the next section when we come to 
discuss the extension to quasi-regular Sasakian manifolds, or rather their associated Fano K\"ahler-Einstein orbifolds. 

However, there are two classes of examples in which necessary and sufficient conditions are known. 
The first is the classification of  Fano K\"ahler-Einstein surfaces already mentioned. One may describe this 
result by saying that a Fano surface admits a K\"ahler-Einstein metric if and only if its Futaki invariant is zero. The 
second class of examples are the \emph{toric} Fano manifolds. Here a complex $(n-1)$-manifold is said to be \emph{toric} if there is a biholomorphic $(\C^*)^{n-1}$ action with a dense open orbit. Then a toric Fano manifold 
admits a K\"ahler-Einstein metric if and only if its Futaki invariant is zero 
\cite{WZh04}. The K\"ahler-Einstein metric is invariant under the 
real torus subgroup $\mathbb{T}^{n-1}\subset (\C^*)^{n-1}$. We shall see that this result generalizes
 to the quasi-regular and irregular Sasakian cases, and so postpone further discussion to later in the article.

Otherwise, examples are somewhat sporadic (as will be the case also in the next section). 
As an example, the 
\emph{Fermat hypersurfaces} $F_{d,n}=\{z_0^d+\cdots+z_n^d=0\}\subset \mathbb{CP}^n$ 
are Fano provided $d\leq n$, and Nadel \cite{Nad90} has shown that these admit K\"ahler-Einstein metrics 
if $n/2\leq d\leq n$. It is straightforward to compute the homology groups of the corresponding 
regular Sasaki-Einstein manifolds $S_{d,n}$ \cite{BG00a}. In particular, in dimension 7 $(n=4)$ one finds 
examples of Sasaki-Einstein 7-manifolds with third Betti numbers $b_3(S_{4,4})=60$, $b_3(S_{3,4})=10$ 
(notice $S_{2,n}$ corresponds to a quadric, and is homogeneous). 

Finally, we stress that in any given dimension there are only finitely many (deformation classes of) Fano manifolds 
\cite{KMM}. Thus there are only finitely many K\"ahler-Einstein structures,
and hence finitely many \emph{regular} Sasaki-Einstein structures, up to continuous deformations of the 
complex structure on the K\"ahler-Einstein manifold. This result is no longer true when one passes 
to the orbifold category, or quasi-regular case.

\section{Quasi-regular Sasaki-Einstein manifolds and hypersurface singularities}\label{sec:quasi}

\subsection{Fano K\"ahler-Einstein orbifolds}

Recall that a quasi-regular Sasakian manifold is a Sasakian manifold whose Reeb foliation has compact leaves, but 
such that the corresponding $U(1)$ action is only locally free, rather than free. As in the previous section, Theorem \ref{thm:quotient} and Proposition \ref{prop:SE} imply that the leaf space of a quasi-regular Sasaki-Einstein manifold is  a compact K\"ahler-Einstein orbifold $(Z,h)$. The main tool in this section will be the converse result obtained using the inversion Theorem 
\ref{thm:inverse}:
\begin{theorem}\label{thm:inversion}
 Let $(Z,h,\omega_Z)$ be a compact simply-connected ($\pi_1^{\mathrm{orb}}(Z)$ trivial) K\"ahler-Einstein orbifold with 
positive Ricci curvature $\mathrm{Ric}_h=2n h$. Let $\pi:S\rightarrow Z$ be the principal $U(1)$ orbibundle 
with first Chern class $-c_1(Z)/I(Z)\in H^2_{\mathrm{orb}}(Z,\Z)$.
 Then $(S,g=\pi^*h + \eta\otimes\eta)$ is a compact simply-connected quasi-regular Sasaki-Einstein 
orbifold, where $\eta$ is the connection 1-form on $S$ with curvature $\diff\eta = 2\pi^*\omega_Z$. 
 Furthermore, if all the 
local uniformizing groups inject into $U(1)$  then the total space $S$ is a smooth manifold.
\end{theorem}
Here the \emph{orbifold Fano index} $I(Z)$ is defined in a precisely analogous way 
to the manifold case: it is the largest positive integer such that $c_1(Z)/I(Z)$ defines 
an integral class in the orbifold Picard group $H^2_{\mathrm{orb}}(Z,\Z)\cap H^{1,1}(Z,\R)$. 
As in the regular case, the principal $U(1)$ orbibundle appearing here is that associated to the 
complex line orbibundle $\left(K_Z^{\mathrm{orb}}\right)^{1/I(Z)}$.

Here we make the important remark that canonical line orbibundle of $Z$, $K_Z^{\mathrm{orb}}$, 
is not necessarily the same as the canonical line bundle defined in the algebro-geometric sense. 
The difference between the two lies in the fact that complex codimension one orbifold singularities 
are not seen by the canonical line bundle, owing to the simple fact that $\C/\Z_m\cong \C$ as an algebraic
variety. More specifically, let the complex codimension one singularities of $Z$ be along divisors 
$D_i$, and suppose that $D_i$ has multiplicity $m_i$ in the above sense. In particular, a K\"ahler-Einstein orbifold metric 
on $Z$ will have a $2\pi/m_i$ conical singularity along $D_i$. Then
\begin{equation}\nonumber
K_Z^{\mathrm{orb}} = K_Z + \sum_i \left(1-\frac{1}{m_i}\right)D_i~.
\end{equation}
The $D_i$ are known as the \emph{ramification divisors}. A Fano K\"ahler-Einstein orbifold is then Fano
in the sense that $\left(K_Z^{\mathrm{orb}}\right)^{-1}$ is positive, which is not the same condition as $K_Z^{-1}$ being positive. 
Also, the Fano indices in the two senses will not in general agree. 

\subsection{The join operation}

As already mentioned, the first examples of quasi-regular Sasaki-Einstein manifolds 
were the quasi-regular 3-Sasakian manifolds constructed in \cite{BGM94}. The first 
examples of quasi-regular Sasaki-Einstein manifolds that are \emph{not} 
3-Sasakian were in fact constructed using these examples, together with the following 
Theorem of \cite{BG00a}:
\begin{theorem}
Let $\mathcal{S}_1$, $\mathcal{S}_2$ be two simply-connected quasi-regular 
Sasaki-Einstein manifolds of dimensions $2n_1-1$, $2n_2-1$, respectively. 
Then there is a natural operation called the \emph{join} which produces in general a 
simply-connected quasi-regular Sasaki-Einstein orbifold $\mathcal{S}_1\star \mathcal{S}_2$ of dimension 
$2(n_1+n_2)-3$. The join is a smooth manifold if and only if 
\begin{equation}\label{shaved}
\mathrm{gcd}\left(\mathrm{ord}(Z_1)l_2,\mathrm{ord}(Z_2)l_1\right)=1~,
\end{equation}
where $l_i=I(Z_i)/\mathrm{gcd}(I(Z_1),I(Z_2))$ are the relative orbifold Fano indices 
of the K\"ahler-Einstein leaf spaces $Z_1$, $Z_2$.
\end{theorem}
Recall here that $\mathrm{ord}(M)$ denotes the order of $M$ as an orbifold (see section \ref{sec:regularity}). 

The proof of this result follows from the simple observation that given two K\"ahler-Einstein 
orbifolds $(Z_1,h_1)$, $(Z_2,h_2)$, the product $Z_1\times Z_2$ carries a direct product 
K\"ahler-Einstein metric which is the sum of $h_1$ and $h_2$, after an appropriate constant rescaling of each. 
The join is then the unique simply-connected Sasaki-Einstein orbifold obtained by applying 
the inversion Theorem \ref{thm:inversion}. The smoothness condition (\ref{shaved}) 
is simply a rewriting of the condition that the local uniformizing groups inject into 
$U(1)$, given that this is true for each of $\mathcal{S}_1$, $\mathcal{S}_2$. In particular, note that if $\mathcal{S}_1$ is a 
regular Sasaki-Einstein manifold ($\mathrm{ord}(Z_1)=1$) and $I(Z_2)$ divides $I(Z_1)$ then the join 
is smooth whatever the value of the order of $Z_2$. 

The join construction can produce interesting non-trivial examples. For example, the 
homogeneous Sasaki-Einstein manifold in (2) of Theorem \ref{thm:regular5} is 
simply $S^3\star S^3$, with the round metric on each $S^3$. More importantly, the join of a quasi-regular 
3-Sasakian manifold with a regular Sasaki-Einstein manifold (such as $S^3$) gives rise to a 
quasi-regular Sasaki-Einstein manifold by the observation at the end of the previous 
paragraph. However, this particular construction produces new examples of Sasaki-Einstein manifolds 
only in dimension 9 and higher. 

\subsection{The continuity method for K\"ahler-Einstein orbifolds}\label{sec:cont}

The Proposition \ref{prop:MA} holds also for compact K\"ahler orbifolds, with an identical proof. 
Thus also in the orbifold category, to find a K\"ahler-Einstein metric on a Fano orbifold 
one must similarly solve the Monge-Amp\`ere equation (\ref{MA}). Of course, since 
necessary and sufficient algebraic conditions on $Z$ are not even known in the smooth manifold case, 
for orbifolds the pragmatic approach of Boyer, Galicki, Koll\'ar and their collaborators is to find a sufficient condition in appropriate classes of examples. Also as in the smooth manifold case, one can use the continuity method 
to great effect. Our discussion in the remainder of this section will closely follow \cite{BG05}.

Suppose, without loss of generality, that we are seeking a solution to (\ref{MA})  with $\kappa=1$. 
Then the continuity method works here by introducing the more general equation
\begin{equation}\label{continuity}
\frac{\det \left(h_{i\bar{j}}+\frac{\partial^2\phi}{\partial z_i\partial \bar{z}_j}\right)}{\det h_{i\bar{j}}} = 
\mathrm{e}^{f-t\phi}~,
\end{equation}
where now $t\in [0,1]$ is a constant parameter. We wish to solve the equation with $t=1$. 
We know from Yau's proof of the Calabi conjecture \cite{Yau78} that there is a solution with $t=0$. 
The classic continuity  argument works by trying to show that the subset of $[0,1]$ where solutions exist is 
both open and closed. In fact openness is a straightforward application of the implicit function theorem. 
On the other hand, closedness is equivalent to the integrals
\begin{equation}\nonumber
\int_Z \mathrm{e}^{-\gamma t\phi_t}\omega_0^{n-1}
\end{equation}
being uniformly bounded, for any constant $\gamma\in ((n-1)/n,1)$. Here $\omega_0$ denotes the K\"ahler form 
for $h_0$, the metric given by Yau's result. Nadel interprets this condition in terms of
\emph{multiplier ideal sheaves} \cite{Nad90}.

This was first studied in the case of Fano orbifolds by Demailly-Koll\'ar \cite{DK01}, and indeed it is their 
results that led to the first examples of quasi-regular Sasaki-Einstein 5-manifolds in \cite{BG01}. 
The key result is the following:
\begin{theorem}\label{thm:DK}
Let $Z$ be a compact Fano orbifold of dimension $\dim_\C Z=n-1$. Then the continuity method 
produces a K\"ahler-Einstein orbifold metric on $Z$ if there is a $\gamma>(n-1)/n$ such that 
for every $s\geq 1$ and for every holomorphic section $\tau_s\in H^0\left(Z,\left(K_Z^{\mathrm{orb}}\right)^{-s}\right)$ 
\begin{equation}\label{DKcondition}
\int_Z |\tau_s|^{-2\gamma/s}\, \omega_0^{n-1}<\infty~.
\end{equation}
\end{theorem}
For appropriate classes of examples, the condition (\ref{DKcondition}) is not too difficult to check. 
We next introduce such a class.

\subsection{Links of weighted homogeneous hypersurface singularities}\label{sec:links}

Let $w_i\in \Z_{>0}$, $i=0,\ldots,n$, be a set of positive integers. We regard these as a vector 
$\mathbf{w}\in (\Z_{>0})^{n+1}$. There is an associated 
 weighted $\C^*$ action on $\C^{n+1}$ given by 
\begin{equation}\label{weightedC}
\C^{n+1} \ni (z_0,\ldots,z_n) \mapsto \left(\lambda^{w_0}z_0,\ldots,\lambda^{w_n}z_n\right)~,
\end{equation}
where $\lambda\in\C^*$ and the $w_i$ are referred to as the \emph{weights}. Without loss of generality 
one can assume that $\mathrm{gcd}(w_0,\ldots,w_n)=1$, so that the $\C^*$ action is effective, although this 
is not necessary.

\begin{definition}
A polynomial $F\in \C[z_0,\ldots,z_n]$ is said to be a \emph{weighted homogeneous polynomial with 
weights $\mathbf{w}$ and degree $d\in\Z_{>0}$} if
\begin{equation}\nonumber
F\left(\lambda^{w_0}z_0,\ldots,\lambda^{w_n}z_n\right)=\lambda^d F(z_0,\ldots,z_n)~.
\end{equation}
\end{definition}

We shall always assume that $F$ is chosen so that the affine algebraic variety
\begin{equation}\label{isolated}
X_F = \{F=0\}\subset \C^{n+1}
\end{equation}
is smooth everywhere except at the origin in $\C^{n+1}$. 

\begin{definition}\label{def:link}
The hypersurface given by (\ref{isolated}) is called a \emph{quasi-homogeneous hypersurface 
singularity}. The \emph{link} $L_F$ of the singularity is defined to be
\begin{equation}\label{link}
L_F = \{F=0\}\cap S^{2n+1}~,
\end{equation}
where $S^{2n+1}=\left\{\sum_{i=0}^n |z_i|^2=1\right\}\subset\C^{n+1}$ is the unit sphere.
\end{definition}

$L_F$ is a smooth $(2n-1)$-dimensional manifold, and it is a classic result of Milnor \cite{Mil} that $L_F$ is $(n-2)$-connected. Indeed, the homology groups of $L_F$ were computed in \cite{Mil, MO70} in terms of the 
so-called \emph{monodromy map}. 

A particularly nice such set of singularities are the so-called Brieskorn-Pham singularities. These take the particular form
\begin{equation}\label{BP}
F=\sum_{i=0}^n z_i^{a_i}~,
\end{equation}
where $\mathbf{a}\in (\Z_{>0})^{n+1}$. Thus the weights of the $\C^*$ action are $w_i=d/a_i$ where $d=\mathrm{lcm}\{a_i\}$. The corresponding 
hypersurface singularities are always isolated, as is easy to check. In this case 
it is convenient to denote the link by $L_F=L(\mathbf{a})$. 
Moreover, to the vector $\mathbf{a}$ 
one associates a graph $G(\mathbf{a})$ with $n+1$ vertices labelled by the $a_i$. 
Two vertices $a_i$, $a_j$ are connected if and only if $\mathrm{gcd}(a_i,a_j)>1$. 
We denote the connected component of $G(\mathbf{a})$ determined by the even integers 
by $C_{\mathrm{even}}$; all even integer vertices are contained in $C_{\mathrm{even}}$, 
although $C_{\mathrm{even}}$ may of course contain odd integer vertices also. The following result is due 
to Brieskorn \cite{Bri66} (although see also \cite{Dim92}):
\begin{theorem}\label{thm:Brieskorn}
The following are true:
\begin{enumerate}
\item The link $L(\mathbf{a})$ is a rational homology sphere if and only if 
either $G(\mathbf{a})$ contains at least one isolated point, or $C_{\mathrm{even}}$ has an odd number of 
vertices and for any distinct $a_i,a_i\in C_{\mathrm{even}}$, $\mathrm{gcd}(a_i,a_j)=2$.
\item The link $L(\mathbf{a})$ is an integral homology sphere if and only if either $G(\mathbf{a})$ 
contains at least two isolated points, or $G(\mathbf{a})$ contains one isolated point and 
$C_{\mathrm{even}}$ has an odd number of vertices and $a_i,a_j\in C_{\mathrm{even}}$ implies 
$\mathrm{gcd}(a_i,a_j)=2$ for any distinct $i,j$.
\end{enumerate}
\end{theorem}
In particular, this result says which $\mathbf{a}$ lead to homotopy spheres $L(\mathbf{a})$. 
Again, classical results going back to Milnor \cite{Mil56} and Smale \cite{Sm61} 
show that in every dimension greater than 4 the differentiable homotopy spheres 
form an Abelian group, where the group operation is given by the connected sum. 
There is a subgroup consisting of those which bound parallelizable manifolds, and these 
groups are known in every dimension \cite{KM63}. They are distinguished by the signature 
$\tau$ of a parallelizable manifold whose boundary is the homotopy sphere. There is a natural 
choice in fact (the Milnor fibre in Milnor's fibration theorem \cite{Mil}), and Brieskorn 
computed the signature in terms of a combinatorial formula involving the $\{a_i\}$. Many of the 
results quoted in the introduction involving rational homology spheres and exotic spheres 
are proven this way. Let us give a simple example:
\begin{example}
By Theorem \ref{thm:Brieskorn} the link $L(6k-1,3,2,2,2)$ is a homotopy 7-sphere. Using Brieskorn's formula 
one can compute the signature of the associated Milnor fibre, with the upshot being
that all 28 oriented diffeomorphism classes on $S^7$ are realized by taking $k=1,2,\ldots,28$.
\end{example}

Returning to the general case in Definition \ref{def:link}, the fact that 
$L_F$ has a natural Sasakian structure was observed as long ago as reference \cite{Tak78}. 
We begin by noting that $\C^{n+1}$ (minus the origin) has a K\"ahler cone metric that is a cone with respect to the 
\emph{weighted} Euler vector field
\begin{equation}\label{weightedEuler}
\tilde{r}\partial_{\tilde{r}} = \sum_{i=0}^n w_i \rho_i\partial_{\rho_i}~,
\end{equation}
where $z_i=\rho_i\exp(\ii \theta_i)$, $i=0,\ldots,n$. The K\"ahler form is $\tfrac{1}{2}\ii \partial\bar\partial \tilde{r}^2$ where 
$\tilde{r}^2$ is a homogeneous degree 2 function under (\ref{weightedEuler}), and a natural choice is 
$\tilde{r}^2=\sum_{i=0}^n \rho_i^{2/w_i}$. The holomorphic vector field $\left(\ii\, 1 + J\right)\tilde{r}\partial_{\tilde{r}}$ of course generates the weighted $\C^*$ action (\ref{weightedC}), and by construction the hypersurface $X_F$ is invariant 
under this $\C^*$ action. Thus the K\"ahler metric inherited by $X_F$ via its embedding (\ref{isolated}) 
is also a K\"ahler cone with respect to this $\C^*$ action, which in turn gives rise to a Sasakian structure 
on $L_F$. 

On the other hand, the quotient of $\C^{n+1}\setminus\{0\}$  by the weighted $\C^*$ action 
is by definition the weighted projective space $\mathbb{P}(\mathbf{w})=\mathbb{CP}^n_{[w_0,\ldots,w_n]}$. 
There is a corresponding commutative square
\begin{equation}
\begin{array}{cccc}
L_F & \longrightarrow & S^{2n+1}&\\
  \bigg\downarrow{\pi} && \bigg\downarrow{} &\\
   Z_F & \longrightarrow & \mathbb{P}(\mathbf{w}),&
\end{array}
\end{equation}
where the horizontal arrows are Sasakian and K\"ahlerian embeddings, respectively, and the vertical 
arrows are orbifold Riemannian submersions. Here $Z_F$ is simply the hypersurface $\{F=0\}$, now 
regarded as defined in the weighted projective space, so $Z_F=\{F=0\}\subset\mathbb{P}(\mathbf{w})$. 
Thus $Z_F$ is a weighted projective variety.

We are of course interested in the case in which $Z_F$ is Fano:
\begin{proposition}\label{prop:Fano}
The orbifold $Z_F$ is Fano if and only if $|\mathbf{w}| - d >0$, where $|\mathbf{w}|=\sum_{i=0}^n w_i$. 
\end{proposition}
This was proven in \cite{BGK05}, but a simpler method of proof \cite{GMSY07} that bypasses the orbifold subtleties is to use Proposition 
\ref{prop:equivGor}. Indeed, $X_F$ is Gorenstein since the smooth locus $X_F\setminus\{0\}$ 
is equipped with a nowhere zero holomorphic $(n,0)$-form, given explicitly in a coordinate chart in which 
the denominator is nowhere zero by
\begin{equation}\label{Omega}
\Omega = \frac{\diff z_1\wedge\cdots\wedge\diff z_n}{\partial F/\partial z_0}~.
\end{equation}
One has similar expressions in charts in which $\partial F/\partial z_i\neq 0$, and it is straightforward to check 
that these glue together into a global holomorphic volume form on $X_F\setminus\{0\}$. Since $\tilde{r}\partial_{\tilde{r}} z_i = w_i z_i$ and 
$F$ has degree $d$, it follows from (\ref{Omega}) that
\begin{equation}\label{Omegacharge}
\mathcal{L}_{\tilde{r}\partial_{\tilde{r}}}\Omega = \left(|\mathbf{w}|-d\right) \Omega~.
\end{equation}
As in the proof of Proposition \ref{prop:equivGor}, positivity of $|\mathbf{w}|-d$ is then equivalent 
to the Ricci form on $Z_F$ being positive. Indeed, 
via a $D$-homothetic transformation we may define $r=\tilde{r}^a$ to be a new K\"ahler potential, 
where $a=(|\mathbf{w}|-d)/n$, so that the new Reeb vector field is 
\begin{equation}\nonumber
\xi = \frac{n}{|\mathbf{w}|-d}\, \tilde{\xi}~,
\end{equation}
and $\mathcal{L}_\xi \Omega = \ii n \Omega$. The resulting K\"ahler metric on $Z_F$ now satisfies $[\rho_Z]=2n[\omega_Z]\in H^{1,1}(Z,\R)$. One may then ask when this metric can be 
deformed to a K\"ahler-Einstein metric, thus giving a quasi-regular Sasaki-Einstein metric on $L_F$. Although in general necessary and sufficient conditions 
are not known, Theorem \ref{thm:DK} gives a sufficient condition that is practical to 
check.

\subsection{Quasi-regular Sasaki-Einstein metrics on links}

Theorem \ref{thm:DK} was used by Demailly-Koll\'ar in their paper \cite{DK01} to prove the existence of 
K\"ahler-Einstein orbifold metrics on certain orbifold del Pezzo surfaces, realized 
as weighted hypersurfaces in $\mathbb{CP}^3_{[w_0,w_1,w_2,w_3]}$. More precisely, they 
produced precisely 3 such examples. The very first examples of 
quasi-regular Sasaki-Einstein 
5-manifolds were constructed in \cite{BG01} using this result, together 
with the inversion Theorem \ref{thm:inversion}. The differential topology 
of the corresponding links can be analyzed using the results described in the previous section, 
together with Smale's Theorem \ref{thm:smale}, resulting in 2 non-regular Sasaki-Einstein metrics on 
$S^2\times S^3$, and 1 on $\# 2\left(S^2\times S^3\right)$. 
An avalanche of similar results followed \cite{Ara02, BG03, BGN02a, BGN03,
BGN03c, JK01a, JK01b}, classifying all such log del Pezzo surfaces (Fano orbifold surfaces 
with only isolated orbifold singularities) for which Theorem \ref{thm:DK} produces 
a K\"ahler-Einstein orbifold metric. This led to quasi-regular Sasaki-Einstein 
structures on $\# k\left(S^2\times S^3\right)$ for all $1\leq k\leq 9$. Compare to Theorem 
\ref{thm:regular5}.

In \cite{BGK05} Theorem \ref{thm:DK} was applied to the Brieskorn-Pham links $L(\mathbf{a})$, giving 
the following remarkable result:
\begin{theorem}\label{thm:BGK}
Let $L(\mathbf{a})$ be a Brieskorn-Pham link, with weighted homogeneous polynomial given by (\ref{BP}). 
Denote $c_i=\mathrm{lcm}(a_0,\ldots,\hat{a}_i,\ldots,a_n)$, $b_i=\mathrm{gcd}(a_i,c_i)$, where 
as usual a hat denotes omission of the entry. Then $L(\mathbf{a})$ admits a quasi-regular Sasaki-Einstein 
metric if the following conditions hold:
\begin{enumerate}
\item $\sum_{i=0}^n \frac{1}{a_i}>1$~,
\item $\sum_{i=0}^n \frac{1}{a_i}< 1 + \frac{n}{n-1}\mathrm{min}_i \{\frac{1}{a_i}\}$,
\item $\sum_{i=0}^n \frac{1}{a_i}< 1 + \frac{n}{n-1}\mathrm{min}_{i,j}\{\frac{1}{b_ib_j}\}$.
\end{enumerate}
The isometry group of the Sasaki-Einstein manifold is then 1-dimensional, generated by the Reeb vector field.
\end{theorem}
The condition (1) is simply the Fano condition in Proposition \ref{prop:Fano}, rewritten in terms of the
$a_i=d/w_i$. Conditions (2) and (3) result from the condition in Theorem \ref{thm:DK}. 
The statement about the isometry group follows from the fact that the K\"ahler orbifolds $Z(\mathbf{a})$ 
have no continuous automorphisms. Combining this with the Matsushima result \cite{Mat57}, which 
goes through in exactly the same way for orbifolds, implies that the isometry group of the K\"ahler-Einstein metric on $Z(\mathbf{a})$ is finite. 

An important modification of this result follows from perturbing the defining polynomial $F$ to 
\begin{equation}\label{deform}
F(\mathbf{a},p) = \sum_{i=0}^n z_i^{a_i} + p(z_0,\ldots,z_n)~,
\end{equation}
where $p$ is a weighted homogeneous polynomial of degree $d$. Then the above Theorem holds also 
for the links of $X_F=\{F(\mathbf{a},p)=0\}\subset \C^{n+1}$, provided the intersection of 
$X_F$ with any number of hyperplanes $\{z_i=0\}$ are all smooth outside the origin.
The polynomials $p$ may depend on complex parameters, which will then lead to continuous families 
of quasi-regular Sasaki-Einstein manifolds in which there is a corresponding family of complex 
structures on $X_F$ or $Z_F$. In fact similar remarks apply  also to the log del Pezzo examples 
that produce non-regular Sasaki-Einstein structures on $\# k\left(S^2\times S^3\right)$ for $1\leq k\leq 9$.

Theorem \ref{thm:BGK} gives only necessary conditions for existence; it is not expected that this result is 
sharp. On the other hand, subsequent work of Ghigi-Koll\'ar \cite{GK07}, combined with the Lichnerowicz obstruction 
of \cite{GMSY07} that we will describe in section \ref{sec:obstructions}, leads to the following:
\begin{theorem}\label{thm:GK}
Let $L(\mathbf{a})$ be a Brieskorn-Pham link such that the $\{a_i\}$ are pairwise relatively prime. 
Then $L(\mathbf{a})$ is homeomorphic to  $S^{2n-1}$ and admits a Sasaki-Einstein metric if and only if 
\begin{equation}\nonumber
1<\sum_{i=0}^n \frac{1}{a_i}<1+n\, \mathrm{min}_i\left\{\frac{1}{a_i}\right\}~.
\end{equation}
\end{theorem}
Of course, the topological statement here follows immediately from the Brieskorn Theorem \ref{thm:Brieskorn}

Using Theorems \ref{thm:BGK} and \ref{thm:GK} it is now straightforward to construct vast numbers of new 
quasi-regular Sasaki-Einstein manifolds by simply finding those $\mathbf{a}$ which satisfy the given inequalities. 
For example, for fixed dimension $n$ one can show that there are only finitely many such $\mathbf{a}$ 
 that, via the Brieskorn result, give rise to homotopy spheres. With the aid of a computer 
one can easily list all such examples in dimensions 5 and 7.
\begin{example}\label{ex:dim5}
It is simple to check that the links $L(2,3,7,k)$ are homotopy spheres for any $k$ that is relatively prime to 
at least 2 of $\{2, 3, 7\}$. Moreover, for $5\leq k\leq 41$ these satisfy the conditions in Theorem \ref{thm:BGK}, 
giving 27 examples of quasi-regular Sasaki-Einstein metrics on $S^5$.  
\end{example}
There are other examples also, given in \cite{BGK05}, some of which have deformations in the sense of 
(\ref{deform}). There are 12 further examples satisfying 
the conditions in Theorem \ref{thm:GK} of the form $L(2,3,5,k)$ for appropriate $k$. Combining all these results \cite{BG07} gives:
\begin{corollary}\label{thm:S5}
There are at least 80 inequivalent families of Sasaki-Einstein structures on $S^5$. Some of these admit continuous 
Sasaki-Einstein deformations, the largest known of which depends on 5 complex parameters.
\end{corollary}
In dimension 7 the authors of \cite{BGKT05} showed that there are 8610 inequivalent families of 
Sasaki-Einstein structures on homotopy 7-spheres given by Theorem \ref{thm:BGK}, and moreover 
there are examples for each of the 28 oriented diffeomorphism classes. More examples 
are produced using Theorem \ref{thm:GK}. In particular:
\begin{example}\label{ex:S7}
The link $L(2,3,7,43,1333)$ is diffeomorphic to the standard 7-sphere. It admits a complex 41-dimensional 
family of Sasaki-Einstein deformations.
\end{example}
Summarizing these results \cite{BG07} gives:
\begin{corollary} Each of the 28 oriented diffeomorphism classes on $S^7$ admit several 
hundred inequivalent families of Sasaki-Einstein structures. In each class, some of these admit 
continuous deformations. The largest such family is that in Example \ref{ex:S7}.
\end{corollary}
Similar results hold also in higher dimensions, although the numbers of solutions grows 
rapidly, prohibiting a complete list even in dimension 9. However, one can show in particular that 
Sasaki-Einstein structures exist, in vast numbers, on both the standard and exotic Kervaire spheres in 
every dimension of the form $4m+1$. We refer the reader to \cite{BGK05} for details.

In \cite{Kol07} Koll\'ar considered K\"ahler-Einstein orbifold metrics on orbifolds whose singularities 
are purely of complex codimension 1. Thus as algebraic varieties these are smooth surfaces. 
This allows for considerably more complicated topology than the log del Pezzo surfaces 
described at the beginning of this section, and leads to:
\begin{theorem}\label{thm:Kollar}
For every $k\geq 6$ there are infinitely many complex $(k-1)$-dimensional families of Sasaki-Einstein structures 
on $\# k \left(S^2\times S^3\right)$.
\end{theorem}
Thus all of the manifolds in Theorem \ref{thm:smale} admit Sasaki-Einstein structures. We 
shall give a very different proof of this in section \ref{sec:toric}, for which the Sasaki-Einstein 
manifolds have an isometry group containing the torus $\mathbb{T}^3$.

We turn next to rational homology spheres. We begin with an example:
\begin{example}\label{ex:rational}
In every dimension the links $L(m,m,.\ldots,m,k)$ are rational homology spheres provided $\mathrm{gcd}(m,k)=1$. 
The conditions in Theorem \ref{thm:BGK} are satisfied if $k>m(m-1)$. The homology groups 
$H_{m-1}(L,\Z)$ are torsion groups of order $k^{b_{m-1}}$ where $b_{m-1}$ is the $(m-1)$th Betti number 
of the link $L(m,m,\ldots,m)$ \cite{BG06}.
\end{example}
The torsion groups here may be computed using \cite{Ran75}. This leads to the following result of \cite{BG06}:
\begin{corollary}
In every odd dimension greater than 3 there are infinitely many smooth compact simply-connected 
rational homology spheres admiting Sasaki-Einstein structures.
\end{corollary}

In dimension 5 the classification Theorem \ref{thm:smale} was extended to general oriented simply-connected 
5-manifolds by Barden \cite{Bar65}. With the additional assumption of being spin, such 5-manifolds 
fall into 3 classes: rational homology spheres; connected sums of $S^2\times S^3$, as in Theorem \ref{thm:smale}; 
and connected sums of these first two classes. Which of these diffeomorphism types of general simply-connected spin 5-manifolds 
admit Sasaki-Einstein structures has been investigated, with the most recent results appearing in 
\cite{BN10}. In particular, there is the following interesting result of Koll\'ar \cite{Kol05}:
\begin{theorem} Let $S$ be a simply-connected 5-manifold admitting a transverse Fano Sasakian structure. 
Then $H_2(S,\Z)_{\mathrm{tor}}$ is isomorphic to one of the following groups:
\begin{equation}
\Z_m^2~, \ \Z_2^{2n}~,  \ \Z_3^4~, \ \Z_3^6~, \ \Z_3^8~, \ \Z_4^4~, \ \Z_5^4~, \quad m\geq 1~, \  n>1~.\nonumber
\end{equation}
\end{theorem}
In particular, this determines also the possible torsion groups for simply-connected Sasaki-Einstein 5-manifolds. 
Precisely which of the manifolds in the Smale-Barden classification could admit Sasaki-Einstein structures 
is listed in Corolloary 11.4.14 of the monograph \cite{BG07}, together with those for which existence 
has been shown. Further results, again using weighted homogeneous hypersurface singularities, have been 
presented recently in \cite{BN10}.

\section{Explicit constructions}\label{sec:explicit}

\subsection{Cohomogeneity one Sasaki-Einstein 5-manifolds}

In the last section we saw that quasi-regular Sasaki-Einstein structures exist in abundance, in every odd dimension. 
It is important to stress that these are existence results, based on sufficient algebro-geometric conditions 
for solving the Monge-Amp\`ere equation (\ref{MA}) on the orbifold leaf space of a quasi-regular Sasakian manifold that 
is transverse Fano. Indeed, the isometry groups of the Sasaki-Einstein manifolds produced via this method  in 
Theorem \ref{thm:BGK} are as small as possible, 
and this lack of symmetry suggests that it will be difficult to write down solutions in explicit form.

On the other hand, given enough symmetry one might hope to find examples of Sasaki-Einstein manifolds for which the 
metric and Sasakian structure can be written down explicitly in local coordinates. Of course, such examples will be rather special.
We have already mentioned in Theorem \ref{thm:homo} that 
homogeneous Sasaki-Einstein manifolds are classified. The next simplest case, in terms of symmetries, is that of \emph{cohomogeneity one}. By definition this means
there is a compact Lie group $G$ of isometries preserving the Sasakian structure which acts such that the generic orbit 
has real codimension 1. In fact the first \emph{explicit} quasi-regular Sasaki-Einstein 5-manifolds constructed were of this form. 
The construction also gave the very first examples of \emph{irregular} Sasaki-Einstein manifolds, 
which had been conjectured by Cheeger-Tian \cite{ChT94} not to exist. The following result was presented in \cite{GMSW04a}:
\begin{theorem}\label{thm:Ypq}
There exist countably infinitely many Sasaki-Einstein metrics on $S^2\times S^3$, labelled by two positive integers 
$p,q\in\Z_{>0}$, $\mathrm{gcd}(p,q)=1$, $q<p$, given explicitly in  local coordinates by
\begin{eqnarray}
\qquad g &=& \frac{1-y}{6}\left(\diff\theta^2+\sin^2\theta\diff\phi^2\right) + \frac{1}{w(y)q(y)}\diff y^2 + \frac{q(y)}{9}\left(\diff\psi 
- \cos\theta\diff\phi\right)^2\nonumber\\
&&+\, w(y)\left[\diff\alpha + f(y)\left(\diff\psi-\cos\theta\diff\phi\right)\right]^2~,\label{Ypqmetric}
\end{eqnarray}
where 
\begin{eqnarray}
 w(y) &=& \frac{2(a-y^2)}{1-y}~, \quad q(y) = \frac{a-3y^2+2y^3}{a-y^2}~, \quad f(y)= \frac{a-2y+y^2}{6(a-y^2)}~,\nonumber
\end{eqnarray}
and the constant $a=a_{p,q}$ is 
\begin{equation}\label{apq}
a=a_{p,q} = \frac{1}{2}- \frac{(p^2-3q^2)}{4p^3}\sqrt{4p^2-3q^2}~. 
\end{equation}
The manifolds are cohomogeneity one under the isometric action of a Lie group with Lie algebra $\mathtt{su}(2)\oplus \mathtt{u}(1) \oplus \mathtt{u}(1)$.
The Sasakian structures are quasi-regular if and only if $4p^2-3q^2=m^2$, $m\in\Z$; otherwise they are irregular of rank 2. In particular, there are countably infinite numbers of quasi-regular and irregular Sasaki-Einstein structures on $S^2\times S^3$.
\end{theorem}

We discovered these manifolds quite by accident, whilst trying to classify a certain class of supergravity solutions 
\cite{GMSW04c}. It is not too difficult to check that the metric in (\ref{Ypqmetric}) is indeed Sasaki-Einstein, although 
a key point is that the local coordinate system here is not in fact well-adapted to the Sasakian structure. For example, 
the Reeb vector field is $\xi=3\partial_\psi-\tfrac{1}{2}\partial_\alpha$. 
Instead these local coordinates are convenient for analysing when and how this metric extends to a smooth complete metric on a compact manifold. 

The metric in the first line of (\ref{Ypqmetric}) can in fact be shown to be a smooth complete 
metric on $S^2\times S^2$, for any value of the constant $a\in (0,1)$, by taking $\theta\in[0,\pi]$, 
$y\in[y_1, y_2]$, and the coordinates $\phi$ and $\psi$ to be periodic with period $2\pi$. 
Here $y_1<y_2$ are the two smallest roots of the cubic appearing in the numerator of the function $q(y)$,
the condition that $a\in (0,1)$ guaranteeing in particular that these roots are real. Geometrically, these coordinates
naturally describe a 4-manifold which is given by the 1-point compactifications of the fibres of the tangent 
bundle of $S^2$, $TS^2$. This results in an 
$S^2$ bundle over $S^2$ that is topologically trivial.
There is a natural action of $SO(3)\times U(1)$,  under which the metric is invariant, in which $SO(3)$ acts in the obvious way on $TS^2$ and the
$U(1)$ acts on the fibre, the latter $U(1)$ being generated by the Killing vector field $\partial_\psi$.
The 1-form in square brackets appearing in the second 
line of (\ref{Ypqmetric}) can then be shown to be proportional to a connection 1-form on the total 
space of a principal $U(1)$ bundle over $S^2\times S^2$, provided $a=a_{p,q}$ is given by 
(\ref{apq}). The integers $p$ and $q$ are simply the Chern numbers of this 
$U(1)$ bundle, so naturally $(p,q)\in \Z\oplus\Z\cong H^2(S^2\times S^2,\Z)$. It is important 
here that $w(y)>0$ for all $y\in [y_1,y_2]$. It is also important
to stress that this principal $U(1)$ bundle is \emph{not} generated by the Reeb vector field, and 
indeed the metric on $S^2\times S^2$ in the first line of (\ref{Ypqmetric}) is neither K\"ahler nor Einstein. 
Via the Gysin sequence and Smale's Theorem \ref{thm:smale}, the total space is diffeomorphic to $S^2\times S^3$ provided 
$\mathrm{gcd}(p,q)=1$. 

In a sense, the Sasaki-Einstein manifolds of Theorem \ref{thm:Ypq} interpolate between the two 5-dimensional 
homogeneous Sasaki-Einstein manifolds given in (1) and (2) of Theorem \ref{thm:regular5}. 
More precisely, setting $p=q$ leads to a Sasaki-Einstein orbifold $S^5/\Z_{2p}$, with the round metric on $S^5$, while $q=0$ instead leads to a Sasaki-Einstein orbifold which is a non-freely acting $\Z_p$ quotient of the homogeneous 
Sasaki-Einstein metric on $V_2(\mathbb{R}^4)$. 
%This is easiest to see using the methods described in section \ref{sec:toric}.

As stated in Theorem \ref{thm:Ypq}, the resulting Sasaki-Einstein manifolds are cohomogeneity one under 
the effective isometric action of a compact Lie group $G$ with Lie algebra $\mathtt{su}(2)\oplus \mathtt{u}(1)\oplus \mathtt{u}(1)$. 
In fact we have the following classification result \cite{Con07}:
\begin{theorem}\label{thm:conti}
Let $(S,g)$ be a compact simply-connected Sasaki-Einstein 5-manifold for which the isometry group acts 
with cohomogeneity one. Then $(S,g)$ is isometric to one of the manifolds in Theorem \ref{thm:Ypq}.
\end{theorem}
Much is known about the structure of cohomogeneity one manifolds, and also
the Einstein equations in this case; a review  was presented in a previous article in this journal series \cite{Wang}.
The cohomogeneity one assumption reduces the conditions for having a complete $G$-invariant Sasaki-Einstein metric to 
solving a system of 
ordinary differential equations on an interval, with certain boundary conditions at the endpoints of this interval. 
Here the interval is parametrized by distance $t$ along a geodesic transverse to a generic orbit of $G$. Denoting the 
stabilizer group of a point on a generic orbit by $H\subset G$, then the manifold $S$ has a dense open subset 
that is equivariantly diffeomorphic to $(t_0,t_1)\times G/H$. At the boundaries $t=t_0$, $t=t_1$ of the interval 
the generic orbit collapses to 2 special orbits $G/H_0$, $G/H_1$. For this to happen smoothly, $H_1/H$ and 
$H_2/H$ must both be diffeomorphic to spheres of positive dimension. Choosing an $\mathrm{Ad}_H$-invariant 
decomposition $\mathfrak{g}=\mathfrak{h}+\mathfrak{m}$, a $G$-invariant metric 
on $S$ is determined by a map from $[t_0,t_1]\rightarrow S^2_+(\mathfrak{m})^H$, where the latter 
is the space of $\mathrm{Ad}_H$-invariant symmetric positive bilinear maps on $\mathfrak{m}$. 
This has appropriate boundary conditions at $t=t_0$ and $t=t_1$ 
that guarantee the metric compactifies to a smooth metric on $S$. We refer the reader to \cite{Con07} for a complete discussion in 
the case of Sasaki-Einstein 5-manifolds. For the cohomogeneity one manifolds in Theorem \ref{thm:Ypq} the 
special orbits are located at $y=y_1$, $y=y_2$.

\subsection{A higher dimensional generalization}

Being Sasaki-Einstein, the Reeb foliation $\mathcal{F}_\xi$ for the manifolds in Theorem \ref{thm:Ypq} 
is transversely K\"ahler-Einstein. The K\"ahler-Einstein metric on a local leaf space is, after an appropriate local
change of coordinates, given by
\begin{eqnarray}\label{YpqKE}
g^T &= &\frac{1-y}{6}\left(\diff\theta^2+\sin^2\theta\diff\phi^2\right) + \frac{1}{w(y)q(y)}\diff y^2\\
&&+ \, \frac{w(y)q(y)}{4}\left(\diff\gamma 
+ \tfrac{1}{3}\cos\theta\diff\phi\right)^2~.\nonumber
\end{eqnarray}
This local K\"ahler metric is of \emph{Calabi} form. By definition, the Calabi ansatz \cite{Cal79} takes a K\"ahler manifold 
$(V,g_V,\omega_V)$ of complex dimension $m$ and produces a local K\"ahler metric in complex dimension 
$m+1$ given by
\begin{eqnarray}\label{Calabi}
h &=& (\beta-y)g_V + \frac{1}{4Y(y)}\diff y^2 + Y(y)\left(\diff\gamma+A\right)^2~,\\
\omega_h &= &(\beta-y)\omega_V - \frac{1}{2}\diff y\wedge \left(\diff\gamma+A\right)~.\nonumber
\end{eqnarray}
Here $A$ is a local 1-form on $V$ with $\diff A = 2\omega_V$, $Y$ is an arbitrary function, and $\beta$ is a 
constant. If $V$ is compact of course the 1-form $A$ cannot be globally defined on $V$. However, if 
$(V,\omega_V)$ is a Hodge manifold then by definition $\omega_V$ is proportional to the curvature 2-form of a 
Hermitian line bundle over $V$. In this case the 1-form $\diff \gamma +A$ may be interpreted globally as 
being proportional to the connection 1-form on the total space $P$ of the associated principal $U(1)$ bundle.  
The K\"ahler metric (\ref{Calabi}) is then defined on $(y_1,y_2)\times P$, where the function $Y(y)$ is 
strictly positive on this interval and $y_2<\beta$ so that $(\beta-y)$ is also strictly positive. 

The K\"ahler-Einstein metric (\ref{YpqKE}) is locally of this form, where one takes $(V,g_V)$ to be the 
standard Fubini-Study K\"ahler-Einstein metric on $\mathbb{CP}^1$, normalized to have volume $2\pi/3$, $4Y(y)=w(y)q(y)$, and 
$\beta=1$. The metric is also cohomogeneity one, where the isometry group has Lie algebra 
$\mathtt{su}(2)\oplus\mathtt{u}(1)$. In fact this metric was constructed as early as \cite{GP79}, 
with some global properties being discussed in \cite{Ped85}. 
 One can replace $(V,g_V)=(\mathbb{CP}^1,g_{
\mathrm{Fubini}-\mathrm{Study}})$ by a general Fano K\"ahler-Einstein manifold of complex dimension $m$. 
The metric $h$ in (\ref{Calabi}) is then itself K\"ahler-Einstein provided $Y$ satisfies an appropriate 
ordinary differential equation. Remarkably, this equation can be solved explicitly in every dimension 
\cite{BB82, PP87}, leading to a local 1-parameter family of K\"ahler-Einstein metrics. 
 However, in the latter reference it was shown that this local metric extends 
to a complete metric on a compact Fano manifold only for a particular member of this family, which is then
simply the homogeneous Fubini-Study metric on $\mathbb{CP}^{m+1}$. 
This family of local metrics was 
subsequently forgotten.

Given the above discussion, it is perhaps then unsurprising that the construction of Sasaki-Einstein manifolds in Theorem \ref{thm:Ypq} extends to 
a construction of infinitely many Sasaki-Einstein manifolds, in every odd dimension $2n-1\geq 5$, for every complete Fano K\"ahler-Einstein manifold $(V,g_V)$ with $\dim_\C V=n-2$. 
The following result was shown in \cite{GMSW04b}, although the more precise statement given here appeared later in 
\cite{MS09}:
\begin{theorem}\label{thm:Ypk}
Let $(V,g_V)$ be a complete Fano K\"ahler-Einstein manifold of complex dimension $\dim_\C V=n-2$ 
with Fano index $I=I(V)$. Then for every choice of positive integers $p,k\in\Z_{>0}$ satisfying 
$Ip/2< k <pI$, $\mathrm{gcd}(p,k)=1$, there is an associated explicit complete simply-connected Sasaki-Einstein 
manifold in dimension $2n-1$.
\end{theorem}
Theorem \ref{thm:Ypq} is the special case in which $(V,g_V)=(\mathbb{CP}^1,g_{
\mathrm{Fubini}-\mathrm{Study}})$ and where $k=p+q$. The proof is almost identical to the proof of 
Theorem \ref{thm:Ypq}. One uses the local 1-parameter family of K\"ahler-Einstein metrics of \cite{BB82, PP87} 
to write down a local 1-parameter family of Sasaki-Einstein metrics in dimension $2n-1$. 
This is a local version of the inversion theorem: given a 
(local) K\"ahler-Einstein metric $h$ with positive Ricci curvature $\mathrm{Ric}_h=2nh$ the local metric
\begin{equation}\label{locallift}
g = h + (\diff\psi+A)^2~,
\end{equation}
is Sasaki-Einstein, where $A$ is a local 1-form with $\diff A = 2\omega_h$.
After an appropriate change of 
local coordinates, one sees that this local metric can be made into a complete metric on a compact 
manifold for a countably infinite number of members of the family. This manifold is the total space of 
a principal $U(1)$ bundle over a manifold which is itself an $S^2$ bundle over $V$. The latter is 
obtained from $K_V^{-1}$ by compactifying each fibre, and the integers $p$, $k$ specify the 
first Chern class of the principal $U(1)$ bundle. Unlike the case $n=3$ in Theorem \ref{thm:Ypq}, 
the homology groups of these Sasaki-Einstein manifolds in general depend on $p$ and $k$. 
Determining which of the Sasakian structures are quasi-regular and which are irregular 
is equivalent to determining whether a certain polynomial of degree $n-1$, with integer 
coefficients depending on $p$ and $k$, has a certain root which is rational or 
irrational, respectively \cite{MS08}. 

\subsection{Transverse Hamiltonian 2-forms}\label{sec:Ham}

In \cite{GMSW05, CLPV05} it was noted that the last construction may be extended 
further by replacing the K\"ahler-Einstein manifold $(V,g_V)$ by a finite set of Fano K\"ahler-Einstein
manifolds $(V_a,g_a)$, $a=1,\ldots,N$, and correspondingly extending Calabi's K\"ahler metric ansatz.
Moreover, in \cite{CLPP05} (see also \cite{MS05}) an infinite set of explicit \emph{cohomogeneity two} 
Sasaki-Einstein metrics were presented on $S^2\times S^3$. These have isometry group 
$\mathbb{T}^3$. In \cite{MS05} it was realized that there is a single 
geometric structure that underlies all of these explicit constructions of Sasaki-Einstein 
manifolds, namely a transverse Hamiltonian 2-form on the K\"ahler leaf space of the Reeb foliation. 

Hamiltonian 2-forms were introduced in \cite{ACG06}:
\begin{definition}
Let $(Z,h,\omega, J)$ be a K\"ahler manifold. A \emph{Hamiltonian 
2-form} $\phi$ is a real $(1,1)$-form that solves non-trivially the equation 
\cite{ACG06}
\begin{equation}\nonumber
\nabla_Y \phi = \frac{1}{2}\left(\diff \, \mathrm{Tr}_{\omega}\phi 
\wedge JY^{\flat} - \diff^c \, \mathrm{Tr}_{\omega}\phi\wedge Y^{\flat}\right)~.
\end{equation}
Here $Y$ is any vector field, $\nabla$ denotes the Levi-Civita connection of $h$, 
and $Y^{\flat}=h(Y,\cdot)$ is the 1-form dual to $Y$.
\end{definition} 
In fact Hamiltonian 2-forms on K\"ahler manifolds are related to another structure that is perhaps 
rather more well-known, especially to relativists:
\begin{proposition}
A $(1,1)$-form $\phi$ is Hamiltonian if and only if $\phi+(\mathrm{Tr}_\omega\phi)\, \omega$ is closed 
and the symmetric 2-tensor $S=J(\phi-(\mathrm{Tr}_\omega \phi)\, \omega)$ is a \emph{Killing tensor}; 
that is, $\mathrm{Sym}\, \nabla S=0$ $($in components, $\nabla_{(i}S_{jk)}=0)$.
\end{proposition}
The proof is an elementary calculation and may be found in \cite{ACG06}. In fact 
if $\phi$ is Hamiltonian then the $(1,1)$-form $\phi - \tfrac{1}{2}(\mathrm{Tr}_\omega\phi)\, \omega$ 
a is conformal Killing 2-form in the sense of \cite{Sem}. In the relativity literature 
such a form is also called a \emph{conformal Killing-Yano} form. Again, this leads to an equivalence between 
conformal Killing 2-forms of type $(1,1)$ and Hamiltonian 2-forms. 
Conformal Killing tensors and forms generalize the notion of conformal Killing vectors. 
The latter generate symmetries of the metric, and the same is true also of Killing tensors, albeit 
in a more subtle way.  For example, a classic early result was that a Killing form gives 
rise to a quadratic first integral of the geodesic equation \cite{PW70}.

The key result about Hamiltonian 2-forms is that 
their existence leads to a very specific form for the K\"ahler metric $h$. Moreover, particularly relevant for us
is that the K\"ahler-Einstein condition is then equivalent to solving 
a simple set of decoupled ordinary differential equations. Below we just sketch how this 
works, referring the reader to  \cite{ACG06} for details. 
We note that many of the resulting ans\"atze for K\"ahler metrics had been 
arrived at prior to the work of \cite{ACG06}, both 
in the mathematics literature (as pointed out in \cite{ACG06}), 
and also in the physics literature. The theory of Hamiltonian 2-forms 
unifies these various approaches.

One first notes that if $\phi$ is a Hamiltonian 2-form, then 
so is $\phi_t = \phi - t\omega$ for any $t\in\mathbb{R}$. One then 
defines the \emph{momentum polynomial} of $\phi$ to be
\begin{equation}\nonumber
p(t) =  \frac{(-1)^m}{m!} *\phi_t^m~.
\end{equation}
Here $m$ is the complex dimension of the K\"ahler manifold and $*$ is 
the Hodge operator with respect to the metric $h$. It is then 
straightforward to show that $\{p(t)\}$ are a set of 
Poisson-commuting Hamiltonian functions for the 1-parameter family of 
Killing vector fields $K(t)=J\, \mathrm{grad}_h p(t)$. For a given point 
in the K\"ahler manifold, these Killing vectors will span a vector 
subspace of the tangent space of the point; the maximum dimension of 
this subspace, taken over all points, is called the \emph{order} $s$ of 
$\phi$. This leads to a local Hamiltonian $\mathbb{T}^s$ action 
on the K\"ahler 
manifold, and one may take a (local) K\"ahler quotient by 
this torus action. The reduced K\"ahler space is a direct product of $N$ K\"ahler 
manifolds that depends on the 
moment map level at which one reduces, but only very weakly. The $2s$-dimensional 
fibres 
turn out to be \emph{orthotoric}, which is a rather special type  
of toric K\"ahler structure. For further details, we refer the reader to 
reference \cite{ACG06}. However, the above should give some idea of how one arrives 
at the following structure theorem of \cite{ACG06}:
\begin{theorem}\label{thm:Ham}
Let $(Z,h,\omega,J)$ be a K\"ahler manifold of complex dimension $m$ with a Hamiltonian 2-form $\phi$ of 
order $s$. This means that the momentum polynomial $p(t)$ has $s$ non-constant roots 
$y_1,\ldots,y_s$. Denote the remaining \emph{distinct} constant roots by $\zeta_1,\ldots,\zeta_N$, 
where $\zeta_a$ has multiplicity $m_a$, so that $p(t)=p_{\mathrm{nc}}(t)p_{\mathrm{c}}(t)$ 
where $p_{\mathrm{nc}}(t)=\prod_{i=1}^s (t-y_i)$ and $p_{\mathrm{c}}(t)=\prod_{a=1}^N 
(t-\zeta_a)^{m_a}$. Then there are functions $F_1,\ldots,F_s$ of one variable such that on a dense open subset 
the K\"ahler structure may be written
\begin{eqnarray}\nonumber
h &=& \sum_{a=1}^N p_{\mathrm{nc}}(\zeta_a)g_a+\sum_{i=1}^s \left[\frac{p'(y_i)}{F_i(y_i)}\diff y_i^2 + 
\frac{F_i(y_i)}{p'(y_i)}\left(\sum_{j=1}^s\sigma_{j-1}(\hat{y}_i)\theta_j\right)^2\right]~.\\
\omega & = & \sum_{a=1}^N p_{\mathrm{nc}}(\zeta_a)\omega_a+\sum_{i=1}^s \diff\sigma_i\wedge\theta_i~, 
\qquad \diff\theta_i = \sum_{a=1}^N (-1)^i\zeta_a^{s-i}\omega_a~.\nonumber
\end{eqnarray}
Here $\sigma_i$ denotes the $i$th elementary symmetric function of the non-constant roots $y_1,\ldots,y_s$, 
and $\sigma_{j-1}(\hat{y}_i)$ denotes the $(j-1)$th elementary symmetric function of the $s-1$ roots 
$\{y_k\mid k\neq i\}$. Moreover, $(g_a,\omega_a)$ is a positive (or negative) definite K\"ahler metric on 
a manifold $V_a$ with $\dim_\C V_a=m_a$.
\end{theorem}
In fact the Hamiltonian 2-form is simply
\begin{equation}\nonumber
\phi = \sum_{a=1}^N \zeta_a p_{\mathrm{nc}}(\zeta_a)\omega_a + \sum_{i=1}^s \left(\sigma_i\diff\sigma_1-\diff\sigma_{i+1}\right)\wedge\theta_i~,
\end{equation}
where $\sigma_{s+1}\equiv 0$. What is remarkable about this ansatz for a K\"ahler structure is the following \cite{ACG06}:
\begin{proposition}\label{prop:Feqn} The K\"ahler metric in Theorem \ref{thm:Ham} is K\"ahler-Einstein if for all $i=1,\ldots,s$ the functions
$F_i$ satisfy
\begin{equation}\label{Feqn}
F'_i(t) = p_{\mathrm{c}}(t)\sum_{j=0}^s b_j t^{s-j}~,
\end{equation}
where $b_j$ are arbitrary constants $($independent of $i\, )$, and for all $a=1,\ldots,N$ $\pm (g_a,\omega_a)$ is K\"ahler-Einstein with scalar curvature
\begin{equation}\nonumber
\mathrm{Scal}_{\pm g_a}= \mp m_a \sum_{i=0}^s b_i\zeta_a^{s-i}~.
\end{equation}
In this case the Ricci form is $\rho_h=-\tfrac{1}{2}b_0\omega$.
\end{proposition}
Of course, this result follows from direct local calculations. Notice that (\ref{Feqn}) may immediately be integrated 
to obtain a local K\"ahler-Einstein metric that is completely explicit, up to the K\"ahler metrics $g_a$.

By taking $m=n-1$ and $b_0=-4n$ one can lift such local K\"ahler-Einstein metrics of positive Ricci curvature 
to local Sasaki-Einstein metrics in dimension $2n-1$ using (\ref{locallift}).  One may then ask when this local 
metric extends to a complete metric on a compact manifold.

In fact all known (or at least known to the author) explicit constructions of Sasaki-Einstein manifolds
are of this form. The Sasaki-Einstein manifolds in Theorem \ref{thm:Ypk} are constructed this way, with 
$s=1$, $N=1$. Indeed, this case is precisely the Calabi ansatz (\ref{Calabi}) for the local K\"ahler-Einstein metric, as already mentioned. 
The generalization in \cite{GMSW05, CLPV05} mentioned at the beginning of this section is $s=1$ but $N\geq 1$. Finally, most 
interesting is to take $s>1$. In particular, for a Sasaki-Einstein 5-manifold this means that necessarily $N=1$ 
and moreover $m_1=0$. In other words, the transverse K\"ahler-Einstein metric is \emph{orthotoric} 
in the sense of reference \cite{ACG06}. The results described in this section give the explicit local form 
of such a metric, although it must be stressed that this is not how they were first derived. In fact in 
\cite{CLPP05} the local family of orthotoric K\"ahler-Einstein metrics was obtained by taking a certain limit of 
a family of black hole metrics. These black hole solutions themselves possess Killing tensors. 
On the other hand, in \cite{MS05} the same local metrics were obtained by taking a limit of the Plebanski-Demianski
metrics \cite{PD}, again a result in general relativity.
It is then simply a matter of analyzing when these local metrics extend to complete metrics on 
a compact manifold. The result is the following \cite{CLPP05, MS05}:
\begin{theorem}\label{thm:Labc}
There exist a countably infinite number of explicit Sasaki-Einstein metric on $S^2\times S^3$, labelled naturally 
by 3 positive integers $a,b,c\in\Z_{>0}$ with $a\leq b$, $c\leq b$, $d=a+b-c$, $\mathrm{gcd}(a,b,c,d)=1$, 
and also such that each of the pair $\{a,b\}$ is coprime to each of $\{c,d\}$. For $a=p-q$, $b=p+q$, $c=p$ 
these reduce to the Sasaki-Einstein structures in Theorem \ref{thm:Ypq}. Otherwise they are cohomogeneity two with 
isometry group $\mathbb{T}^3$.
\end{theorem}
The proof here is rather different to that of the proof of Theorem \ref{thm:Ypq}. In fact it is easiest to understand 
the global structure using the toric methods developed in the next section. For integers $\{a,b,c\}$ not satisfying 
some of the coprime conditions one obtains Sasaki-Einstein orbifolds.  
The conditions under which the Sasakian structures are quasi-regular is not simple to determine explicitly in general, and involves 
a quartic Diophantine equation. Generically one expects them to be irregular. The next section allows one to characterize the Sasaki-Einstein manifolds in Theorem \ref{thm:Labc}: they are all of the simply-connected toric Sasaki-Einstein manifolds with second Betti number $b_2(S)=1$. 

\section{Toric Sasaki-Einstein manifolds}\label{sec:toric}

\subsection{Toric Sasakian geometry} 
We begin with the following:
\begin{definition}\label{def:toric} A Sasakian manifold $(S,g)$ is said to be \emph{toric} 
if there is an effective, holomorphic and Hamiltonian action 
of the torus $\T^n$ on the corresponding K\"ahler cone $(C(S),\bar{g},\omega,J)$ with 
Reeb vector field $\xi\in \mathfrak{t}_n=$ Lie algebra of $\T^n$. \end{definition}
Here we have used the same symbol for an element of the Lie algebra $\mathfrak{t}_n$ and the 
corresponding vector field on $C(S)$ induced by the group action. The abuse of notation should not cause 
confusion as the meaning should always be clear. 
The Hamiltonian property means that $\T^n$ acts on constant $r$ surfaces in $C(S)$. This in turn implies that
there exists a $\T^n$-invariant moment map 
\begin{equation}\label{moment}
\mu:C(S)\rightarrow \mathfrak{t}^*_n~, \qquad \mbox{where} \quad \langle \mu,\zeta\rangle = \tfrac{1}{2}r^2 \eta(\zeta)~, \quad \forall 
\zeta\in \mathfrak{t}_n~.
\end{equation}
Definition \ref{def:toric} is taken from \cite{MSY06}, and is compatible
with the earlier definition of toric contact manifold appearing in \cite{BG00b} (up to a factor of $\tfrac{1}{2}$ in the moment map definition).

The condition on the Reeb vector field, $\xi\in\mathfrak{t}_n$,  implies
that the image $\mu(C(S))\cup \{0\}$ is a \emph{strictly convex rational polyhedral cone} $\mathcal{C}^*\subset \mathfrak{t}_n^*$ \cite{FdeMT97, Ler}. (Toric symplectic cones with Reeb vector fields not satisfying this condition form a short list and
have been classified \cite{Ler}.) By definition this means that 
$\mathcal{C}^*$ may be presented as
\begin{equation}\label{poly}
\mathcal{C}^* = \{y\in \mathfrak{t}_n^*\mid \langle y,v_a\rangle \geq 0~, \ \  a=1,\ldots,d\}\subset \mathfrak{t}_n^*~.\end{equation}
Here the rationality condition means that $v_a\in \Z_{\T^n}\equiv \ker \{\exp: 
\mathfrak{t}_n\rightarrow \mathbb{T}^n\}$. On choosing a basis this means that we may think of $v_a\in\Z^n\subset\R^n\cong \mathfrak{t}_n$, and without loss of generality we  assume that 
the $\{v_a\}$ are primitive. We also assume that the set $\{v_a\}$ is minimal, in the sense that 
removing any $v_a$ from the definition in (\ref{poly}) would change the polyhedral cone $\mathcal{C}^*$. The strictly convex condition means that $\mathcal{C}^*$ is a cone over a compact
convex polytope of dimension $n-1$. It follows that necessarily the number of bounding hyperplanes is $d\geq n$. 

We denote by $\mathrm{Int}\, \mathcal{C}^*$ the open interior of $\mathcal{C}^*$. The $\T^n$ action on 
$\mu^{-1}\left(\mathrm{Int}\, \mathcal{C}^*\right)$ is free, and moreover the latter is a Lagrangian 
torus fibration over $\mathrm{Int}\, \mathcal{C}^*$. On the other hand, the bounding faces 
(called \emph{facets}) $\{\langle y,v_a\rangle=0\}\cap \mathcal{C}^*$ lift to $\T^{n-1}$-invariant complex codimension 
one submanifolds of $C(S)$ that are fixed point sets of the $U(1)\cong\T\subset \T^n$ subgroup 
specified by $v_a\in\Z_{\T^n}$.

The image $\mu(S)=\mu(\{1\}\times S\subset C(S))$ is easily seen from (\ref{moment}) to be 
\begin{equation}\nonumber
\mu(S) = \{y\in\mathcal{C}^*\mid \langle y,\xi\rangle = \tfrac{1}{2}\}~.
\end{equation}
Here the hyperplane $\{y\in\mathfrak{t}_n^*\mid \langle y, \xi\rangle=\tfrac{1}{2}\}\subset\mathfrak{t}_n^*$ is called the
\emph{characteristic hyperplane} \cite{BG00b}.
This intersects the moment cone $\mathcal{C}^*$ to form a 
compact $n$-dimensional polytope $\Delta({\xi})=\mu(\{r\leq 1\})$, bounded by $\partial\mathcal{C}^*$ 
and a $(n-1)$-dimensional compact convex polytope $H(\xi)$
which is the image $\mu(S)$ of the Sasakian manifold $S$  in $\mathfrak{t}_n^*$. Since $\mu(\xi)=\tfrac{1}{2}r^2>0$ on 
$C(S)$ this immediately implies that 
the Reeb vector field $\xi\in\mathrm{Int}\, \mathcal{C}$ where 
\begin{equation}\nonumber
\mathcal{C}=\{\xi\in\mathfrak{t}_n\mid \langle y,\xi\rangle\geq 0~,\ \  \forall y\in\mathcal{C}^*\}\subset\mathfrak{t}_n\end{equation}
 is the \emph{dual cone} to $\mathcal{C}^*$. This is also 
a convex rational polyhedral cone by Farkas' Theorem. 

Recall that in section \ref{sec:Einstein} we explained that the space
$X=C(S)\cup \{r=0\}$ can be made into a complex analytic space in a unique way. 
For a toric Sasakian manifold in fact $X$ is an affine toric variety; that is, $X$ is an affine variety 
equipped with an effective holomorphic action of the \emph{complex} torus $\mathbb{T}^n_{\mathbb{C}}\cong (\mathbb{C}^*)^n$ which has a dense open orbit. The affine toric variety $X$ may be constructed 
rather explicitly as follows. Given the polyhedral cone $\mathcal{C}^*$ one defines a linear map 
$A:\R^d\rightarrow\R^n$ via $A(e_a)=v_a$, where $\{e_a\}$ denotes the standard orthonormal 
basis of $\R^d$. The strictly convex condition on $\mathcal{C}^*$ implies that $A$ is surjective. 
This induces a corresponding map of tori $\tilde{A}:\T^d\rightarrow\T^n$. The kernel 
$\ker \tilde{A}$ is a compact abelian subgroup of $\T^d$ of rank 
$d-n$ and $\pi_0(\ker \tilde{A})\cong \Z_{\T^n}/\mathrm{span}_{\Z}\{v_a\}$. Then 
the affine variety $X$ is simply $X=\mathrm{Spec}\, \C[z_1,\ldots,z_d]^{\ker\tilde{A}}$, 
the ring of invariants. This is a standard construction in toric geometry \cite{Ful}, and goes by the name of 
Delzant's Theorem.
The fact that $X$ is toric is also clear via this construction: the torus $\T_\C^n\cong \T_\C^d/\ker\tilde{A}_\C$ 
acts holomorphically on $X$ with a dense open orbit. In this algebro-geometric language
the cone $\mathcal{C}$ is precisely the \emph{fan} for the affine toric variety $X$. 

Let $\partial_{\phi_i}$, $i=1,\ldots,n$, be a 
basis for $\mathfrak{t}_n$, where 
$\phi_i\in [0,2\pi)$ are coordinates on the real torus $\mathbb{T}^n$. We then have the following 
very explicit description of the  space of toric Sasakian metrics \cite{MSY06}:
\begin{proposition}\label{class} The space of 
toric K\"ahler cone metrics on $C(S)$ is a product
\begin{equation}\nonumber
\mathrm{Int}\, \mathcal{C}\times \mathcal{H}^1(\mathcal{C}^*)\nonumber\end{equation}
where $\xi\in\mathrm{Int}\, \mathcal{C}\subset\mathfrak{t}_n$ labels the Reeb vector field and 
$\mathcal{H}^1(\mathcal{C}^*)$ denotes the space of homogeneous degree one functions on $\mathcal{C}^*$ that are 
smooth up to the boundary (together with the convexity condition below). 

Explicitly, on the 
dense open image of $\mathbb{T}^n_{\mathbb{C}}$ we have
\begin{equation}\label{toricmetric}
\bar{g} =\sum_{i,j=1}^n G_{ij}\diff y^i\diff y^j + G^{ij}\diff\phi_i\diff\phi_j~,\end{equation}
where $G_{ij}=\partial_{y_i}\partial_{y_j}G$ 
with matrix inverse $G^{ij}$, and the function
\begin{equation}\label{toricdec}
G(y) = G_{\mathrm{can}}(y) + G_{\xi}(y) + \psi(y)\end{equation}
is required to be strictly convex with $\psi(y)\in \mathcal{H}^1(\mathcal{C}^*)$ and
\begin{eqnarray}\nonumber
G_{\mathrm{can}}(y) & =& \frac{1}{2}\sum_{a=1}^d \langle y,v_a\rangle\log \langle y,v_a\rangle~,\nonumber\\
G_{\xi}(y) &= &\frac{1}{2}\langle y,\xi\rangle\log \langle y,\xi\rangle - \frac{1}{2}\left(\sum_{a=1}^d 
\langle y, v_a\rangle\right)\log \left(\sum_{a=1}^d 
\langle y, v_a\rangle\right)~.\nonumber\end{eqnarray}
\end{proposition}
The coordinates $(y_i,\phi_i)$ are called \emph{symplectic toric} coordinates. The $y_i$ are simply the Hamiltonian 
functions for $\partial_{\phi_i}$:
\begin{equation}\nonumber
y_i = \langle \mu, \partial_{\phi_i}\rangle = \tfrac{1}{2}r^2 \eta(\partial_{\phi_i})~, \qquad \omega = \sum_{i=1}^n \diff y_i \wedge \diff \phi_i~.
\end{equation}
The function $G(y)$ is called the \emph{symplectic potential}. 
Setting $G(y)=G_{\mathrm{can}}(y)$ gives precisely the K\"ahler metric on $C(S)$ induced 
via K\"ahler reduction of the flat metric on $\C^d$. That is, $C(S)$ equipped with the metric 
given by (\ref{toricmetric}) and $G(y)=G_{\mathrm{can}}(y)$ is isomorphic to 
the K\"ahler quotient $(\C^d,\omega_{\mathrm{flat}})//\ker\tilde{A}$ 
at level zero. The origin of $\C^d$ here projects to the singular point $\{r=0\}$ in $X$. 
The function $G_{\mathrm{can}}(y)$ 
 has a certain singular behaviour at the boundary $\partial\mathcal{C}^*$ 
of the polyhedral cone. This is required precisely so that the metric compactifies 
to a smooth metric on $C(S)$. 
By construction, the K\"ahler metric $\bar{g}$ in (\ref{toricmetric}) is 
a cone with respect to $r\partial_r=\sum_{i=1}^n 2y_i\partial_{y_i}$. On the other hand, the 
complex structure in these coordinates is 
\begin{equation}\nonumber
J = \left(\begin{array}{cc}0 & -G^{ij}\\ G_{ij} & 0 \end{array}\right)~,
\end{equation}
and one easily checks that $J(r\partial_r)=\sum_{i,j=1}^n 2G_{ij}y_j\partial_{\phi_i}=\xi$, with $\xi$ determined by $G_{\xi}(y)$ in (\ref{toricdec}).
Proposition \ref{class} extends earlier work 
of Guillemin \cite{Gui} and Abreu \cite{Abr} from the K\"ahler case to the Sasakian case.

The following topological result is from \cite{Ler04}:
\begin{proposition}\label{prop:lerman}
Let $\mathcal{S}$ be a toric Sasakian manifold. Then $\pi_1(S)\cong \Z_{\mathbb{T}^n}/\mathrm{span}_{\Z}\{v_a\}$, 
$\pi_2(S)\cong \Z^{d-n}$.
\end{proposition}
In particular, $S$ is simply-connected if and only if the vectors $\{v_a\}$  that define the moment polyhedral cone 
$\mathcal{C}^*$ form a $\Z$-basis of $\Z_{\T^n}\cong\Z^n$. Using the Hurewicz Isomorphism Theorem
and Smale's Theorem \ref{thm:smale} then gives:
\begin{corollary}\label{cor:toricsmale}
Let $\mathcal{S}$ be a simply-connected toric Sasakian 5-manifold. Then $S$ is diffeomorphic to $\#k\left(S^2\times S^3\right)$
where $k=d-n$.
\end{corollary}

Finally, we note that an affine toric variety is $\ell$-Gorenstein in the sense of Definition \ref{def:Gor} 
if and only if there is a basis for the torus $\mathbb{T}^n$ 
for which $v_a=(\ell,w_a)$ for each $a=1,\ldots,d$, and $w_a\in\mathbb{Z}^{n-1}$. 
In particular, for a simply-connected toric Sasaki-Einstein manifold the 
affine toric variety $X$ will be Gorenstein, and hence there will exist a basis 
such that $v_a=(1,w_a)$. 

\begin{example} The Sasaki-Einstein manifolds in Theorem 
\ref{thm:Labc} are toric, and in fact the proof makes it evident that
the corresponding affine toric varieties are 
$X=\mathrm{Spec}\, \C[z_1,z_2,z_3,z_4]^{\C^*(a,b,c)}$, where 
$\C^*(a,b,c)$ is the 1-dimensional subgroup of $(\C^*)^4$ specified by the 
lattice vector $(a,b,-c,-a-b+c)$. The fact that the entries in this vector sum to zero is equivalent to
$X$ being Gorenstein. 
\end{example}

\subsection{Sasaki-Einstein metrics}

Proposition \ref{class} gives a rather explicit description of the space of toric Sasakian metrics 
on the link of an affine toric singularity. We may then ask which of these are Sasaki-Einstein. 
In fact a rather more basic question is for which Reeb vector fields $\xi\in\mathrm{Int}\, \mathcal{C}$ 
is there a Sasaki-Einstein metric. Notice there was no analogous question in the approach of 
section \ref{sec:quasi}: there we had a fixed affine variety, namely a weighted homogeneous hypersurface 
singularity,  with a \emph{fixed} choice of holomorphic Reeb vector field $\xi-\ii J(\xi)$, namely that associated to the weighted $\C^*$ action. 

Without any essential loss of generality, we consider simply-connected Sasakian manifolds. Then we know that 
the corresponding affine toric variety must be Gorenstein if the cone is to admit a 
Ricci-flat K\"ahler cone metric, and hence there is a basis such that $v_a=(1,w_a)$. 
We assume we have chosen such a basis. 

The key idea in \cite{MSY06} was that an Einstein metric $g$ on $S$ with Ricci curvature 
$\mathrm{Ric}_g = 2(n-1)g$ is a critical point of the Einstein-Hilbert action
\begin{equation}\label{EH}
I[g]=\int_S \left[\mathrm{Scal}_{g}+2(n-1)(3-2n)\right]\diff\mu_g~,
\end{equation}
where $\diff\mu_g$ is the Riemannian volume form associated to the metric $g$ and
as earlier $\mathrm{Scal}_{g}$ denotes the scalar curvature. We may then restrict this 
functional to the space of toric Sasakian metrics. The insight in \cite{MSY06} was 
that this functional in fact depends only on the Reeb vector field $\xi$ of the Sasakian 
structure. Direct calculation gives:
\begin{proposition}\label{prop:Z} The Einstein-Hilbert action (\ref{EH}), restricted to the space 
of toric Sasakian metrics on the link of an affine toric Gorenstein singularity, 
induces a function $I:\mathrm{Int}\, \mathcal{C}\rightarrow \R$ given by
\begin{equation}\label{Z}
I(\xi) = 8n(n-1)(2\pi)^n \left[\langle e_1,\xi \rangle - (n-1)\right] \vol(\Delta(\xi))~.
\end{equation}
Here $e_1=(1,0,0,\ldots,0)$ and $\vol(\Delta(\xi))$ denotes the Euclidean volume 
of the polytope $\Delta(\xi)=\mu\left(\{r\leq 1\}\right)$.
\end{proposition}
A toric Sasaki-Einstein metric is a critical point of $I$ defined in (\ref{Z}). 
Of course $\mathcal{C}$ is itself a cone, and one may first take the derivative of 
$I$ along the Euler vector field of this cone. Using the fact that $\vol(\Delta(\xi))$ is 
homogeneous degree $-n$ one easily checks that this derivative is zero if and only if 
$\langle e_1,\xi\rangle = n$. Thus a critical point of $I$ lies on the interior of the intersection of this 
plane with $\mathcal{C}$. Call the latter compact convex polytope $P\subset \mathfrak{t}_n$. 
It follows that the Reeb vector field for a Sasaki-Einstein metric is a critical point of 
\begin{equation}\nonumber
I\mid_{\mathrm{Int}\, P}\, = 8n(n-1)(2\pi)^n \vol(\Delta)~.
\end{equation}
This is, up to a constant of proportionality, also just the \emph{Riemannian} volume of $(S,g)$. 
If we write 
$\xi = \sum_{i=1}^n \xi_i \partial_{\phi_i}$ 
then it is simple to compute
\begin{eqnarray}\label{toricdev1}
\frac{\partial\vol(\Delta)}{\partial\xi_i} & = & \sum_{k=1}^n \frac{1}{2 \xi_k\xi_k}\int_{H({\xi})} 
y^i \, \diff \sigma~,\\ \label{toricdev2}
\frac{\partial^2\vol(\Delta)}{\partial\xi_i\partial \xi_j} &= &\sum_{k=1}^n \frac{2(n+1)}{\xi_k\xi_k}\int_{H({\xi})} 
y^i y^j \, \diff\sigma~.\end{eqnarray}
Here $\diff\sigma$ is the standard measure induced on the $(n-1)$-polytope $H({\xi})=\mu(S)\subset\mathcal{C}^*$. 
Uniqueness and existence of a critical point of $I$ now follows from a standard 
convexity argument: $\vol(\Delta)$ is a strictly convex (by (\ref{toricdev2})) positive function on 
the interior of a compact convex polytope 
$P$. Moreover, $\vol(\Delta)$ diverges to $+\infty$ at $\partial P$. This can be seen rather explicitly 
from the formula for the volume of the polytope $\vol(\Delta)$, but more conceptually 
for $\xi\in\partial\mathcal{C}$ the vector field $\xi$ in fact vanishes somewhere on $C(S)$. 
Specifically, the bounding facets 
of $\mathcal{C}$ correspond to the generating rays of $\mathcal{C}^*$ under the duality map between 
cones; $\xi$ being in a bounding facet of $\mathcal{C}$ implies that the corresponding vector field
then vanishes on the inverse image, under the moment map, of the dual generating ray of $\mathcal{C}^*$. 
It follows from these comments that $I$ 
must have precisely one critical point in the interior of $P$, and we have thus proven:
\begin{theorem}
There exists a unique Reeb vector field $\xi\in\mathrm{Int}\, \mathcal{C}$ for which the toric Sasakian 
structure on the link of an affine toric Gorenstein singularity can be Sasaki-Einstein.
\end{theorem}
Having fixed the Reeb vector field, the problem of finding a Sasaki-Einstein metric now 
reduces to deforming the transverse K\"ahler metric to a transverse K\"ahler-Einstein metric. 
As in the regular and quasi-regular cases, this is a Monge-Amp\`ere problem. 
To analyze this it is more convenient to introduce complex coordinates. 
Recall that the complex torus $\T^n_\C$ is a dense open subset of $C(S)$. 
Introducing log complex coordinates $z_i=x_i+\ii \phi_i$ on $\T^n_\C$, the K\"ahler structure is
\begin{equation}\nonumber
\omega = 2 \ii \partial\bar\partial F~, \qquad \bar{g} =\sum_{i,j=1}^n F_{ij}\diff x_i\diff x_j + F_{ij}\diff \phi_i\diff\phi_j~.
\end{equation}
Here the K\"ahler potential is $F(x)=\tfrac{1}{4}r^2$ and $F_{ij} = \partial_{x_i}\partial_{x_j}F$. This is 
related to the symplectic potential $G$ by Legendre transform
\begin{equation}\nonumber
F(x) = \left(\sum_{i=1}^n y_i\partial_{y_i}G - G\right)(y=\partial_x F)~.
\end{equation}
Having fixed the holomorphic structure on $C(S)$ (this being determined uniquely up to equivariant biholomorphism by $\mathcal{C}^*$)  and fixing the Reeb vector field 
to be the unique critical point of $I$, we may set $\psi=0$ in (\ref{toricdec}) 
to obtain an explicit toric Sasakian metric $g_0$ that is a critical point of $I$. This is our background metric.  
We are then in the situation of Proposition \ref{prop:def}: any other 
Sasakian metric with the same holomorphic structure on the cone and same Reeb vector field 
is related to this metric via a smooth basic function $\phi\in C^\infty_B(S)$. Thus, if 
$g$ is a Sasaki-Einstein metric with this property then 
\begin{equation}\nonumber
\omega^T-\omega^T_0 = \ii \partial_B\bar\partial_B \phi~,
\end{equation}
where $\omega^T_0$ and $\omega^T$ are the transverse K\"ahler forms 
associated to $g_0$ and $g$, respectively. 
The holomorphic volume form on $C(S)$ is \cite{MSY06}
\begin{equation}
\Omega = \mathrm{e}^{x_1+\ii \phi_1}\left(\diff x_1+\ii \diff\phi_1\right)\wedge\cdots\wedge \left(\diff x_n
+\ii \diff\phi_n\right)~,
\end{equation}
and the critical point condition $\langle e_1,\xi\rangle=n$ is equivalent to $\mathcal{L}_\xi \Omega = \ii n \Omega$. 
Thus using Proposition \ref{prop:equivGor} and 
the transverse $\partial\bar\partial$ lemma again we may also write
\begin{equation}
\rho_0^T-2n \omega_0^T = \ii\partial_B\bar{\partial}_B f~,
\end{equation}
with $f\in C^\infty_B(S)$ smooth and basic. Then Proposition \ref{prop:MA} 
goes through in exactly the same way in the transverse sense, with resulting transverse 
Monge-Amp\`ere equation
\begin{equation}\label{transverseMA}
\frac{\det \left(g^T_{0\, i\bar{j}}+\frac{\partial^2\phi}{\partial z_i\partial \bar{z}_j}\right)}{\det g^T_{0\, i\bar{j}}} = 
\mathrm{e}^{f-2n\phi}~,
\end{equation}
where now $z_1,\ldots,z_{n-1}$ are local complex coordinates on the leaf space of the Reeb foliation 
$\mathcal{F}_\xi$.

This problem was recently studied in detail in \cite{FOW09}. In fact the Monge-Amp\`ere problem 
is almost identical to the case of toric K\"ahler-Einstein manifolds studied in \cite{WZh04}. 
The moment polytope in the latter case is essentially replaced by the 
polytope $H(\xi)$ in the Sasakian case. The continuity method is used 
to prove existence, as in section \ref{sec:cont}, 
and crucially the work of \cite{ElK-A90} on extending 
Yau's estimates \cite{Yau78} to transverse Monge-Amp\`ere equations is appealed to to show that 
the $C^0$ estimate for the basic function $\phi$ is sufficient to solve the equation. 
Thus the main step is to prove the 
 $C^0$ estimate for $\phi$, and this closely follows the proof in the K\"ahler-Einstein case 
\cite{WZh04}. The result of \cite{FOW09, CFO08} is the following:
\begin{theorem}\label{thm:FOW}
There exists a unique toric Sasaki-Einstein metric on the link of any affine toric Gorenstein singularity.
\end{theorem}
Here uniqueness was proven in \cite{CFO08}, and is understood up to automorphisms 
of the transverse holomorphic structure. Thus the existence and uniqueness problem for toric Sasaki-Einstein manifolds is completely solved. 
We note that in \cite{CFO08} the authors stated this theorem with the weaker requirement that the affine toric singularity is 
$\ell$-Gorenstein. For $\ell>1$ the links of such singularities will not be simply-connected, although 
the converse is not true. However, more importantly the existence of a Killing spinor implies the Gorenstein 
condition, as discussed in section \ref{sec:spinors}, which is why we have presented Theorem \ref{thm:FOW} this way. 
If one does not care for the existence of a Killing spinor, one can weaken the Gorenstein condition to 
$\mathbb{Q}$-Gorenstein ($\ell$-Gorenstein for some $\ell)$.

Combining Theorem \ref{thm:FOW} with the topological result of Corollory \ref{cor:toricsmale} leads to:
\begin{corollary} There exist infinitely many toric Sasaki-Einstein structures on 
$\# k\left(S^2\times S^3\right)$, for every $k\geq 1$.
\end{corollary}
In fact we have now done enough to see that the explicit metrics in Theorem \ref{thm:Labc} 
are all of the toric Sasaki-Einstein metrics on $S^2\times S^3$.

\begin{example} We comment on two particularly interesting examples. The Sasaki-Einstein structure on 
$S^2\times S^3$ with $p=2$, $q=1$ in Theorem \ref{thm:Ypq} has 
corresponding affine variety $X=\mathrm{Spec}\, \C[z_1,z_2,z_3,z_4]^{\C^*(2,2,-1,-3)}$ \cite{MS06}. 
Using standard toric geometry methods it is straightforward to see 
that $X\setminus \{o\}$ is the total space of the canonical line bundle 
over the 1-point blow up of $\mathbb{CP}^2$, minus the zero section. Equivalently, 
$X$ is obtained from this canonical line bundle by contracting the zero section. 
The latter is a Fano surface which doesn't admit a K\"ahler-Einstein metric, as discussed in section \ref{sec:existence}. 
It is equivalent to say that the canonical choice of holomorphic Reeb vector field 
$\xi-\ii J(\xi)$ that rotates the $\C^*$ fibre over the del Pezzo surface 
cannot be the Reeb vector field for a Sasaki-Einstein metric. Indeed, in this 
example one can easily compute the function $I$ in Proposition \ref{prop:Z} and show that 
this choice of $\xi$ is indeed not a critical point. Instead the critical $\xi$ 
gives an irregular Sasaki-Einstein structure of rank 2 \cite{MSY06}.

In the latter case this irregular Sasaki-Einstein structure associated to the 1-point blow-up 
of $\mathbb{CP}^2$ is completely explicit. For the 2-point blow-up there is no known explicit metric, 
but Theorem \ref{thm:FOW} implies there exists a unique toric Sasaki-Einstein metric on the 
total space of the principal $U(1)$ bundle associated to the canonical line bundle over the 
surface. In \cite{MSY06} the critical Reeb vector field was computed explicitly, showing 
that this is again irregular of rank 2. 
\end{example}

Finally, although Theorem \ref{thm:FOW} settles the existence and uniqueness of toric Sasaki-Einstein manifolds in general, 
we point out that prior to this result van Coevering \cite{vC06} proved
the existence of infinite families of distinct toric quasi-regular Sasaki-Einstein structures on $\# k(S^2\times S^3)$, for each odd $k>1$, using a completely different method. 
He finds certain quasi-regular toric Sasakian submanifolds of 3-Sasakian manifolds 
obtained via the quotient construction mentioned in section \ref{sec:3sas}, and then applies an orbifold generalization of a result of Batyrev-Selivanova  
\cite{BS99} to deform the corresponding K\"ahler orbifold to a K\"ahler-Einstein orbifold. Thus, although 3-Sasakian geometry 
plays a role in this construction, the Sasaki-Einstein metrics are not induced from the 3-Sasakian structure.

\section{Obstructions}\label{sec:obstructions}

\subsection{The transverse Futaki invariant}

A toric Sasaki-Einstein metric has a Reeb vector field $\xi$ which is a critical point of the function 
$I$ in Proposition \ref{prop:Z}. The derivative of $I$ is of course a linear map on a space of holomorphic 
vector fields, and its vanishing is a necessary (and in the toric case also sufficient) condition 
for existence of a Sasaki-Einstein metric, or equivalently a transverse K\"ahler-Einstein metric
for the foliation $\mathcal{F}_\xi$. Given the discussion in section \ref{sec:existence}
it is then not surprising that the derivative of the function $I$ is essentially 
a \emph{transverse Futaki invariant}. This was demonstrated in \cite{MSY07}, although 
our discussion here follows more closely the subsequent treatment in \cite{FOW09}.

Throughout this section we suppose that we have Sasakian structure $\mathcal{S}$ 
with Reeb foliation $\mathcal{F}_\xi$ satisfying $0<c_1^B\in H^{1,1}_B(\mathcal{F}_\xi)$ and 
$0=c_1(D)\in H^2(S,\R)$. Via Proposition \ref{prop:equivGor} it is equivalent 
to say that, after a possible $D$-homothetic transformation, we have $2\pi c_1^B=n[\diff\eta]\in  H^{1,1}_B(\mathcal{F}_\xi)$ . 
Assuming also that $S$ is simply-connected, then by
 Proposition \ref{prop:equivGor} it is also equivalent to 
say that the corresponding Stein space $X=C(S)\cup\{r=0\}$ is Gorenstein 
with $\mathcal{L}_\xi\Omega=\ii n\Omega$, where $\Omega$ is a nowhere zero holomorphic
$(n,0)$-form on $C(S)$.

Following \cite{FOW09} we begin with:
\begin{definition}\label{Hamhol}
A complex vector field $\zeta$ on a Sasakian manifold $\mathcal{S}$ is said to be \emph{Hamiltonian 
holomorphic} if
\begin{enumerate}
\item its projection to each leaf space is a holomorphic vector field; and
\item the complex-valued function $u_\zeta\equiv \tfrac{1}{2} \eta(\zeta)$ satisfies
\begin{equation}\nonumber
\bar\partial_B u_\zeta = -\tfrac{1}{4}i_\zeta\diff\eta~.
\end{equation}
\end{enumerate}
Such a function $u_\zeta$ is called a \emph{Hamiltonian function}.
\end{definition}
If $(x,z_1,\ldots,z_{n-1})$ are coordinates for a local foliation chart $U_\alpha$ then one may write 
\begin{equation}\nonumber
\zeta = \eta(\zeta)\partial_x + \sum_{i=1}^{n-1} \zeta^i\partial_{z_i} - 
\eta\left(\sum_{i=1}^{n-1}\zeta^i\partial_{z_i}\right)\partial_x~,
\end{equation}
where $\zeta^i$ are local basic holomorphic functions. 

It is straightforward to see that $\zeta+\ii\eta(\zeta)r\partial_r$ is 
then a holomorphic vector field on $C(S)$. A Hamiltonian holomorphic vector 
field in the sense of Definition \ref{Hamhol} is precisely the orthogonal projection 
to $S=\{r=1\}$ of a Hamiltonian holomorphic vector field on the K\"ahler cone $(C(S),\bar{g},\omega)$ whose 
Hamiltonian function is basic and homogeneous degree zero under $r\partial_r$. 
As pointed out in \cite{FOW09}, the set of all Hamiltonian holomorphic vector fields 
is a Lie algebra $\mathfrak{h}$. Moreover, if the transverse K\"ahler metric has constant scalar curvature
 then $\mathfrak{h}$ is necessarily reductive \cite{NT88}; this is a transverse generalization of the 
Matsushima result mentioned in section \ref{sec:existence}. Thus the nilpotent radical 
of $\mathfrak{h}$ acts as an obstruction to the existence of a transverse constant scalar curvature K\"ahler metric, 
and in particular a transverse K\"ahler-Einstein metric.

Since $2\pi c_1^B=n[\diff\eta]$, by the transverse $\partial\bar\partial$ lemma there exists a discrepancy potential 
$f\in C^\infty_B(S)$ such that
\begin{equation}\nonumber
\rho^T-n\diff\eta = \ii \partial_B\bar{\partial}_B f~.
\end{equation}
We may then define
\begin{equation}\label{futaki}
\mathscr{F}(\zeta) = \int_S \zeta(f)\, \diff\mu_g~.
\end{equation}
Here the Riemannian measure is 
\begin{equation}\label{volelement}\diff\mu_g = \frac{1}{2^{n-1}(n-1)!}\, \eta\wedge \left(\diff \eta\right)^{n-1}~.
\end{equation} 
Compare this to  the Futaki invariant (\ref{Fut}): in the regular, or quasi-regular, case (\ref{futaki})
precisely reduces to (\ref{Fut}) by integration over the $U(1)$ Reeb fibre, up to an overall proportionality constant.
By following Futaki's original computation \cite{Fut83} it is not difficult to show that $\mathscr{F}(\zeta)$ is independent of 
the transverse K\"ahler metric in the K\"ahler class $[\omega^T]=[\tfrac{1}{2}\diff\eta]\in H^{1,1}_B(\mathcal{F}_\xi)$. 
Thus $\mathscr{F}:\mathfrak{h}\rightarrow \C$ is a linear function on $\mathfrak{h}$ whose non-vanishing obstructs the existence of a 
transverse K\"ahler-Einstein metric in the fixed basic K\"ahler class. This result was extended \cite{FOW09, 
BGS08} to obstructions to the existence of Sasakian metrics with harmonic basic 
$k$th Chern form, again generalizing the K\"ahler result to the transverse setting.

\subsection{The relation to K\"ahler cones}

The results of section \ref{sec:toric} motivated the following set-up in \cite{MSY07}. 
Fix a complex manifold $(C(S)\cong \R_{>0}\times S,J)$ where $S$ is compact, with 
maximal torus $\T^s\subset \mathrm{Aut}(C(S),J)$. Then 
let $\KCM(C(S),J)$ be the space of K\"ahler cone metrics on $(C(S),J)$ 
which are compatible with the complex structure $J$ and such that $\T^s$ acts Hamiltonianly (preserving 
constant $r$ surfaces) with
the Reeb vector field $\xi\in\mathfrak{t}_s=$ Lie algebra of $\T^s$. For each metric 
in $\KCM(C(S),J)$ there is then an associated moment map given by
\begin{equation}\nonumber
\mu:C(S)\rightarrow\mathfrak{t}_s^*~,\qquad \langle \mu,\zeta\rangle = \tfrac{1}{2}r^2\eta(\zeta)~.
\end{equation}
The image is a strictly convex rational polyhedral cone \cite{FdeMT97}, and moreover all these 
cones are isomorphic for any metric in $\KCM(C(S),J)$.  The toric case is when the rank of the torus is maximal, $s=n$. 

For any metric $\bar{g}\in \KCM(C(S),J)$ we may consider the volume functional 
\begin{equation}\label{vol}
\mathrm{Vol}: \KCM(C(S),J)\rightarrow\mathbb{R}_{>0}~, \qquad \mathrm{Vol}(\bar{g})=\int_S \diff\mu_g=\mathrm{vol}(S,g)~.
\end{equation}
Here $g$ is the Sasakian metric on $S$ induced from the K\"ahler cone metric $\bar{g}$. 
Alternatively, it is simple to see  that
\begin{equation}\label{DH}
2^{n-1}(n-1)!\, \mathrm{Vol}(\bar{g}) = \int_{C(S)} \mathrm{e}^{-r^2/2}\, \diff\mu_{\bar{g}} = \int_{C(S)} \mathrm{e}^{-r^2/2} \frac{\omega^n}{n!}~,
\end{equation}
where $\omega$ is the K\"ahler form for $\bar{g}$. Of course, the volume of the cone itself is 
divergent, and the factor of $\exp(-r^2/2)$ here acts as a convergence factor. Since $\tfrac{1}{2}r^2$ 
is the Hamiltonian function for the Reeb vector field $\xi$, the second formula in (\ref{DH}) 
takes the form of a Duistermaat-Heckman integral \cite{DH1, DH2}. Corresponding localization formulae
were discussed in \cite{MSY07}, but we shall not discuss this here.
We then have the following from \cite{MSY07}:
\begin{proposition}\label{prop:vol}
The functional $\mathrm{Vol}$ depends only on the Reeb vector field $\xi$.
\end{proposition}
This perhaps needs some clarification. Via Proposition \ref{prop:def}, any two K\"ahler cone 
metrics $\bar{g}$, $\bar{g}'\in \KCM(C(S),J)$ with the same Reeb vector field $\xi$ have corresponding contact forms related by
\begin{equation}\nonumber
\eta' = \eta + \diff^c_B \phi~,
\end{equation}
where $\phi\in C^\infty_B(S)$ is basic with respect to $\mathcal{F}_\xi$. Notice that the K\"ahler cone metrics 
will have K\"ahler potentials given by smooth functions $r$, $r'$, and that these then give 
different embeddings of $S$ into $C(S)$. Proposition \ref{prop:vol} 
is proven by writing $r'=r\exp(t\phi/2)$ and showing that the derivative of $\mathrm{Vol}$ with respect to $t$ 
at $t=0$ is zero,  independently of the choice of $\phi\in C^\infty_B(S)$.

We may next consider the derivatives of $\mathrm{Vol}$ \cite{MSY07}:
\begin{proposition}\label{volders}
The first and second derivatives of $\mathrm{Vol}$ are given by
\begin{eqnarray}
\label{dev1}
\diff\mathrm{Vol}(\chi)&=& -n\int_S \eta(\chi)\, \diff\mu_g~,\\\label{dev2}
\diff^2\mathrm{Vol}(\chi_1,\chi_2)& =&n(n+1)\int_S \eta(\chi_1)\eta(\chi_2)\, \diff\mu_g~.\end{eqnarray}
\end{proposition}
More formally, what we mean here by (\ref{dev1}) is that we have a 1-parameter family 
$\{\bar{g}(t)\}_{-\epsilon<t<\epsilon}$ of K\"ahler cone metrics in $\KCM(C(S),J)$ with 
$\bar{g}(0)=\bar{g}$ and $\chi=\diff\xi/\diff t\mid_{t=0}$. Then (\ref{dev1}) 
is the derivative of $\mathrm{Vol}$ with respect to $t$ at $t=0$. Similarly for the second derivative
(\ref{dev2}). Of course, these formulae reduce to (\ref{toricdev1}), (\ref{toricdev2}), respectively, in the toric case. 
Also as in that case, the second derivative formula implies that $\mathrm{Vol}$ is formally 
a strictly convex function. Notice that setting $\chi=\xi$ in (\ref{dev1}) implies that 
$\mathrm{Vol}$ is homogeneous degree $-n$ under a $D$-homothetic transformation 
with $\xi\mapsto \lambda\xi$. Indeed, (\ref{scaling}) implies that $\eta\mapsto\lambda^{-1}\eta$ 
and hence the volume element (\ref{volelement}) scales as $\diff\mu_g\mapsto \lambda^{-n}\diff\mu_g$.

Of course our interest is Ricci-flat K\"ahler cones. If $S$ is simply-connected, which we assume for simplicity, then a necessary condition for  $(C(S),J)$ to admit a compatible such metric is that there is a compatible nowhere zero holomorphic $(n,0)$-form $\Omega$. The analytic space $X=C(S)\cup\{r=0\}$ is hence Gorenstein. 
We then introduce the space $\KCM(C(S),\Omega)$ as those $\bar{g}\in \KCM(C(S),J)$ 
such that $\mathcal{L}_\xi \Omega = \ii n\Omega$. 
By Proposition \ref{prop:equivGor}, this 
is equivalent to those $\bar{g}$ for which $[\rho^T]=2n[\omega^T]\in H^{1,1}_B(\mathcal{F}_\xi)$, 
where $\xi$ is the associated Reeb vector field for the K\"ahler cone metric $\bar{g}$. 
As discussed around Proposition \ref{prop:equivGor}, this is a necessary condition for
a Sasaki-Einstein metric.  We then have the following from \cite{MSY07}:
\begin{proposition}\label{genZ}
The Einstein-Hilbert action (\ref{EH}), interpreted as a functional on the space $\KCM(C(S),\Omega)$, 
depends only on the Reeb vector field (in the same sense as Proposition \ref{prop:vol}). Moreover,
\begin{equation}
I(\bar{g}) = 4(n-1)\mathrm{Vol}(\bar{g})~,
\end{equation}
for $\bar{g}\in \KCM(C(S),\Omega)$.
\end{proposition}
Again, the proof is by direct calculation. Since a Ricci-flat K\"ahler cone metric is a critical point of 
$I$, the derivative of $I$ is an obstruction. More precisely, combining the last two Propositions gives:
\begin{corollary}\label{cor:Zcond} Let $\bar{g}\in \KCM(C(S),\Omega)$. Then a necessary condition to be able to 
deform the corresponding Sasakian metric $g$ on $S$ via a transverse K\"ahler deformation to a 
Sasaki-Einstein metric is that
\begin{equation}\label{Zcond}
\int_S \eta(\chi)\, \diff\mu_g = 0
\end{equation}
holds for all $\chi\in\mathfrak{t}_s$ satisfying $\mathcal{L}_\chi\Omega=0$.
\end{corollary}
By assumption, a vector field $\chi\in\mathfrak{t}_s$ is holomorphic on the cone and preserves 
the K\"ahler structure. The condition that it also preserves $\Omega$ guarantees we are varying 
within $\KCM(C(S),\Omega)$, in accordance with Proposition~\ref{genZ}. Since 
$\chi$ preserves $r$ there is a corresponding Hamiltonian function 
$\mu_\chi$ given by $\mu_\chi= \langle\mu,\chi\rangle=\tfrac{1}{2}r^2\eta(\chi)$. 
Since also $\xi\in \mathfrak{t}_s$, in particular $[\xi,\chi]=0$ and it follows that 
$\mu_\chi$ is a basic function of homogeneous degree zero under $r\partial_r$. Thus by 
the comments after Definition \ref{Hamhol}, the orthogonal projection to $S$ of 
$\chi$ is a Hamiltonian holomorphic vector field on $S$. Let us denote this by 
$p_*\chi$ where $p:C(S)\rightarrow S$ is the projection.

In \cite{MSY07} it was shown that for $\chi$ as in Corollary \ref{cor:Zcond}
\begin{equation}\nonumber
\diff\mathrm{Vol}(\chi) = -n\int_S \eta(\chi)\, \diff\mu_g = -\frac{1}{2}\int_S p_*J(\chi)(f)\, \diff \mu_g 
= -\frac{1}{2}\mathscr{F}(p_* J(\chi))~,
\end{equation}
where $\mathscr{F}$ was defined in (\ref{futaki}).
The proof in \cite{MSY07} used spin geometry, while in \cite{FOW09} the same relation was 
obtained without using spinors. Both are quite elementary computations and we refer the 
reader to these references for details.

Thus the Futaki invariant attains an interesting new interpretation when lifted to Sasakian geometry: 
it is essentially the derivative of the Einstein-Hilbert action for Sasakian metrics. 

\subsection{The Bishop and Lichnerowicz obstructions}

We turn now to two further simple obstructions \cite{GMSY07} to the existence of a Ricci-flat K\"ahler cone metric. 
As in the previous section, we fix a complex manifold $(C(S)\cong \R_{>0}\times S,\Omega)$, with $S$ 
compact and $\Omega$ a nowhere zero holomorphic $(n,0)$-form on $C(S)$. For example, we could 
take this to be the smooth part $X_F\setminus\{o\}$ of a weighted homogeneous hypersurface singularity, as in 
section \ref{sec:links}. We will suppose that $\bar{g}\in \KCM(C(S),\Omega)$ is a Ricci-flat K\"ahler cone 
metric on $C(S)$ compatible with this structure, with some Reeb vector field $\xi$, and aim 
to show that under certain conditions we are led to a contradiction. 

Recall from Proposition \ref{prop:vol} that the Riemannian volume of a Sasakian 
manifold induced from a choice of K\"ahler cone metric $\bar{g}\in \KCM(C(S),\Omega)$ depends 
only on the Reeb vector field $\xi$. Here the underlying complex manifold is regarded as fixed. 
We may also understand this as follows:
\begin{proposition}\label{quasi} Suppose that $\bar{g}\in \KCM(C(S),\Omega)$ induces a regular or quasi-regular Sasakian structure on $S$ (simply-connected)  
with K\"ahler leaf space $Z$. Then
\begin{equation}\label{voltop}
\frac{\mathrm{vol}(S,g)}{\mathrm{vol}(S^{2n-1},g_{\mathrm{standard}})} = \frac{I(Z)}{n^n}\int_Z c_1(Z)^{n-1}~,\end{equation}
where $I(Z)\in \Z_{>0}$ is the orbifold Fano index of $Z$  and $(S^{2n-1},g_{\mathrm{standard}})$ is the standard round sphere.\end{proposition}
The condition $\bar{g}\in \KCM(C(S),\Omega)$ implies that, in the regular or quasi-regular case, 
$[\rho_Z]=2n[\omega]_Z\in H^{1,1}(Z,\R)$. The above result then follows since 
$c_1(Z)$ is represented by $\rho_Z/2\pi$. Notice also that we have integrated over the Reeb $U(1)$ fibre. 
Here the simply-connected condition means that the associated complex line orbibundle has 
first Chern class $-c_1(Z)/I(Z)$, as in Theorem \ref{thm:inversion}, which determines the length of the 
generic Reeb $S^1$ fibre. Note that $\mathrm{vol}(S^{2n-1},g_{\mathrm{standard}})=2\pi^n/(n-1)!$.

In the regular or quasi-regular case, the independence of the volume of transverse K\"ahler transformations 
follows simply because the volume of the K\"ahler leaf space depends only on the K\"ahler class. 
On the other hand, this is sufficient to prove the more general statement in Proposition \ref{prop:vol} 
since the space of quasi-regular Reeb vector fields in $\KCM(C(S),\Omega)$ is dense 
in the space of all possible Reeb vector fields. This is simply the statement that 
an irregular Reeb vector field corresponds to an irrational slope vector in 
$\mathfrak{t}_s$, while regular or quasi-regular Reeb vector fields are rational vectors in 
this Lie algebra. The latter are dense of course.

\begin{example}\label{ex:HSvol}
Our main class of examples will be the weighted homogeneous hypersurface singularities 
of section \ref{sec:links}. Thus $C(S)$ is the smooth locus $X_F\setminus\{o\}$, with 
$\Omega$ given by (\ref{Omega}). If $\bar{g}\in \KCM(C(S),\Omega)$ with Reeb vector field 
generating the canonical $U(1)\subset\C^*$ action given by the corresponding weighted $\C^*$ action, then 
one can compute
\begin{equation}\label{HSvol}
\mathrm{vol}(S,g) = \frac{2d}{w(n-1)!}\left(\frac{\pi (|\mathbf{w}|-d)}{n}\right)^n~.
\end{equation}
The reader should consult section \ref{sec:links} for a reminder of the definitions here. 
One can prove (\ref{HSvol}) either by directly using (\ref{voltop}), which was done in 
\cite{BH02} for the case of well-formed orbifolds (where all orbifold singularities have complex codimension at least 2),  
or \cite{GMSY07} using the methods developed in \cite{MSY07}.
\end{example}

We turn now to the related obstruction. 
Bishop's theorem \cite{BC64} implies 
that for any $(2n-1)$-dimensional compact Einstein manifold $(S,g)$ 
with $\mathrm{Ric}_g=2(n-1)g$ we have
\begin{equation}\label{bishinequality}
\mathrm{vol}(S,g)\leq \mathrm{vol}(S^{2n-1},g_{\mathrm{standard}})~.\end{equation}
In the current set-up the left hand side depends only on the holomorphic structure of the cone 
$C(S)$ and Reeb vector field $\xi$. If one picks a Reeb vector field and computes this using, for 
example, (\ref{voltop}), and the result violates (\ref{bishinequality}), then there cannot be 
a Ricci-flat K\"ahler cone metric on $(C(S),\Omega)$ with $\xi$ as Reeb vector field. 

Of course, it is not clear {\it a priori} that 
this condition can ever obstruct existence in this way. However, Example \ref{ex:HSvol} 
shows that the condition is not vacuous. Combining (\ref{HSvol}) with (\ref{bishinequality}) leads immediately to:
\begin{theorem}
The link of a weighted homogeneous hypersurface singularity admits a compatible Sasaki-Einstein structure only if
\begin{equation}\nonumber
d(|\mathbf{w}|-d)^n\leq wn^n~.
\end{equation}
\end{theorem}
It is simple to write down infinitely many examples of weighted homogeneous hypersurface singularities 
that violate this inequality. For example, take $F=z_1^2+z_2^2+z_3^2+z_4^k$. 
For $k$ odd the link is diffeomorphic to $S^5$, while for $k$ even it is diffeomorphic to $S^2\times S^3$. 
The Bishop inequality then obstructs compatible Sasaki-Einstein structures on these links for all $k>20$.

On the other hand,  in \cite{GMSY07} 
we conjectured more generally that for \emph{regular} Reeb vector fields the Bishop inequality never obstructs. 
This is equivalent to the following conjecture about smooth Fano manifolds:
\begin{conjecture}\label{fano} Let $Z$ be a smooth Fano manifold 
of complex dimension $n-1$ with Fano index $I(Z)\in\Z_{>0}$. Then
\begin{equation}\nonumber
I(Z)\int_Z c_1(Z)^{n-1} \leq n \int_{\mathbb{CP}^{n-1}} c_1(\mathbb{CP}^{n-1})^{n-1} = n^n~,
\end{equation}
with equality if and only if $Z=\mathbb{CP}^{n-1}$.
\end{conjecture}
This is related to, although slightly different from, a standard 
conjecture about Fano manifolds. For further details, see \cite{GMSY07}.

We turn next to another obstruction. 
To state this, fix a K\"ahler cone metric $\bar{g}\in \KCM(C(S),\Omega)$ with Reeb vector field $\xi$ and
consider the eigenvalue equation
\begin{equation}\label{charge}
\xi(f) = \ii\lambda  f~.\end{equation}
Here $f:C(S)\rightarrow\C$ is a holomorphic function and we consider $\lambda> 0$. 
The holomorphicity of $f$ implies that $f=r^{\lambda}u$ where 
$u$ is a complex-valued homogeneous degree zero function under $r\partial_r$, or in other words a complex-valued function on $S$. Now on a K\"ahler manifold 
a holomorphic function is in fact \emph{harmonic}. That is 
$\Delta_{\bar{g}}f=0$, where $\Delta_{\bar{g}}$ denotes the Laplacian on $(C(S),\bar{g})$ acting on functions. 
 On the other hand, since $\bar{g}=\diff r^2 + r^2 g$ is a cone 
we have
\begin{equation}\nonumber
 \Delta_{\bar{g}} = r^{-2}\Delta_g - {r^{-2n+1}}\partial_r\left(r^{2n-1}\partial_r\right)~,
\end{equation}
and so (\ref{charge}) implies that
\begin{equation}\nonumber
\Delta_g u = \nu u~,\end{equation}
where
\begin{equation}\label{energy}
\nu= \lambda(\lambda+2(n-1))~.\end{equation}
In other words, a holomorphic function $f$ on $C(S)$ with definite weight $\lambda$, as in (\ref{charge}), 
leads automatically to an eigenfunction of the Laplacian $\Delta_g=\diff^*\diff$ on $(S,g)$ acting on functions.

We again appeal to a classical estimate in Riemmanian geometry. Suppose 
that $(S,g)$ is a compact Einstein manifold with $\mathrm{Ric}_g=2(n-1)g$. The first non-zero 
eigenvalue $\nu_1>0$ of $\Delta_g$ is bounded from below
\begin{equation}\nonumber
\nu_1 \geq 2n-1~.\end{equation}
This is Lichnerowicz's theorem \cite{Li58}. 
Moreover, equality holds if and only if $(S,g)$ is isometric to the round sphere $(S^{2n-1},g_{\mathrm{standard}})$  \cite{Ob62}. 
From (\ref{energy}) we immediately see that for holomorphic functions $f$ on $C(S)$ of 
weight $\lambda$ under $\xi$, Lichnerowicz's 
bound becomes $\lambda\geq 1$. This leads to another potential holomorphic obstruction to 
the existence of Sasaki-Einstein structures. Again, {\it a priori} it is not clear whether or not 
this will ever serve as an obstruction. In fact for \emph{regular} Sasakian structures 
 one can prove \cite{GMSY07} this condition is always trivial. 
This follows from the fact that $I(Z)\leq n$ for any smooth Fano $Z$ of complex 
dimension $n-1$. 

However, there exist plenty of obstructed quasi-regular examples:
\begin{theorem}
The link of a weighted homogeneous hypersurface singularity admits a compatible Sasaki-Einstein structure only if
\begin{equation}\nonumber
|\mathbf{w}|-d\leq nw_{\mathrm{min}}~.
\end{equation}
\end{theorem}
Here $w_{\mathrm{min}}$ is the smallest weight. Moreover, this bound can be saturated if and only if $(C(S),\bar{g})$ is $\C^{n}\setminus\{0\}$ with its flat metric. Notice this result is precisely the necessary direction in
Theorem \ref{thm:GK}.
As another example, consider again the case $F=z_1^2+z_2^2+z_3^2+z_4^k$. 
The coordinate $z_4$ has Reeb weight $\lambda=6/(k+2)$, which obstructs for all $k>4$. For $k=4$ we have $\lambda=1$, 
but since in this case the link is diffeomorphic to $S^2\times S^3$ the Obata result \cite{Ob62} obstructs this marginal case also. 
In fact these examples are interesting because the compact Lie group $SO(3)\times U(1)$ is an automorphism 
group of the complex cone and acts with cohomogeneity 
one on the link. The Matsushima result implies this will be the isometry group of any Sasaki-Einstein metric, 
and then Theorem \ref{thm:conti} in fact also rules out all $k\geq 3$. Indeed, notice that 
$k=1$ corresponds to the round $S^5$ while $k=2$ is the homogeneous Sasaki-Einstein structure 
on $S^2\times S^3$. The reader will find further interesting examples in \cite{GMSY07}.

Very recently it has been shown in \cite{ChSh09} that for weighted homogeneous hypersurface singularities 
the Lichnerowicz condition obstructs if the  Bishop condition obstructs. More precisely:
\begin{theorem}
Let $w_0,\ldots,w_n,d$ be positive real numbers such that 
\begin{equation}
d\left(|\mathbf{w}|-d\right)^n>wn^n~,
\nonumber
\end{equation}
and $d<|\mathbf{w}|$, where $|\mathbf{w}|=\sum_{i=0}^n w_i$, $w=\prod_{i=0}^n w_i$. Then $|\mathbf{w}|-d>nw_{\mathrm{min}}$. 
\end{theorem}
In particular, this shows that Conjecture \ref{fano} is true for smooth Fanos realized as hypersurfaces 
in weighted projective spaces.

As a final comment, we note that more generally the Lichnerowicz obstruction involves 
holomorphic functions on $(C(S),\Omega)$ of small weight with respect to $\xi$,
whereas the Bishop obstruction is a statement about the volume of $(S,g)$, which 
is determined by the asymptotic growth of holomorphic functions on $C(S)$, analogously
to Weyl's asymptotic formula \cite{BH02, GMSY07}.

\section{Outlook}\label{sec:outlook}

We conclude with some brief comments on open problems in Sasaki-Einstein geometry. 
Clearly, one could describe many more.
We list the problems in decreasing order of importance (and difficulty).

In general, the  existence of a K\"ahler-Einstein metric on a Fano manifold is expected to be equivalent 
to an appropriate notion of stability, in the sense of geometric invariant theory. We discussed this briefly 
in section \ref{sec:existence} and described K-stability. The relation between the Futaki invariant and such notions of stability is well-understood. 
Very recently, in the preprint \cite{RT09} the Lichnerowicz obstruction of section \ref{sec:obstructions} was related to 
stability of Fano orbifolds. The stability condition comes from a Kodaira-type embedding into a weighted projective space, 
as opposed to the Kodaira embedding into projective space used for Fano manifolds. This then leads to a notion 
of stability under the automorphisms of this weighted projective space. In particular, 
slope stability leads to some fairly explicit obstructions to the existence of K\"ahler-Einstein metrics (or more 
generally constant scalar curvature K\"ahler metrics) on Fano orbifolds. This includes the Lichnerowicz 
obstruction as a special case, although the role of 
the Bishop obstruction is currently rather more mysterious from this point of view.

A natural question is how to extend these ideas to Sasaki-Einstein geometry in general. In particular, 
how should one understand stability for irregular Sasakian structures? 
\begin{problem}
Develop a theory of stability for Sasakian manifolds that is related to necessary and sufficient conditions 
for existence of a Sasaki-Einstein metric.
\end{problem}
An obvious approach here is to use an idea we have already alluded to in the previous section: one 
might approximate an irregular Sasakian structure using a quasi-regular one, or perhaps 
more precisely a sequence of quasi-regular Sasakian structures that converge to an irregular Sasakian structure 
in an appropriate sense. Then one can use the notions of stability for  orbifolds developed in 
\cite{RT09}. A key difference between the K\"ahler and Sasakian cases, though, is that one 
is free to move the Reeb vector field, which in fact changes the K\"ahler leaf space. 
Given that the Sasakian description of obstructions to the existence of K\"ahler-Einstein orbifold metrics 
led to rather simple differentio-geometric descriptions of these 
obstructions in section \ref{sec:obstructions}, one might also anticipate that embedding 
Sasakian manifolds into spheres might be a beneficial viewpoint. That is, one takes the Sasakian lift 
of the Kodaira-type embeddings encountered in stability theory.

It is just about conceivable that one could classify Sasaki-Einstein manifolds in dimension 5:
\begin{problem}
Classify  simply-connected Sasaki-Einstein 5-manifolds.
\end{problem}
We remind the reader that  this has been done for \emph{regular} Sasaki-Einstein 5-manifolds (Theorem 
\ref{thm:regular5}). This includes the homogeneous Sasaki-Einstein 5-manifolds, and moreover 
cohomogeneity one Sasaki-Einstein manifolds are also  classified by Theorem \ref{thm:conti}. The latter two classes are subsets of the toric Sasaki-Einstein manifolds, which are classified by Theorem \ref{thm:FOW}. 
Indeed, recall that the rank of a Sasakian structure is the dimension of the closure of the 1-parameter subgroup of the isometry group 
generated by the Reeb vector field. If a Sasaki-Einstein 5-manifold has rank 3 then it is toric, and so classified 
in terms of polytopes by Theorem \ref{thm:FOW}. On the other hand, 
rank 1 are regular and quasi-regular. 
Is it possible to state necessary and sufficient conditions for the orbifold leaf space of a transversely Fano Sasakian 5-manifold to admit a K\"ahler-Einstein metric? 
We remind the reader that those simply-connected spin 5-manifolds that can possibly 
admit Sasaki-Einstein structures are listed as a subset of the Smale-Barden classification 
of such 5-manifolds in \cite{BG07}; many, but apparently not all, of these can be realized as 
transversely Fano links of weighted homogeneous hypersurface singularities. The most complete 
discussion of what is known about existence of quasi-regular Sasaki-Einstein structures 
in this case appears in \cite{BN10}. 
 It is rank 2 that is perhaps most problematic. 
In fact, all known rank 2 Sasaki-Einstein 5-manifolds are toric. This leads to the simpler problem of 
whether or not 
there exist rank 2 Sasaki-Einstein 5-manifolds that are not toric. 
This problem is interesting for the reason that
none of the methods for constructing Sasaki-Einstein manifolds 
described in this paper are capable of producing such an example. 

Theorem \ref{thm:Ham} gives a fairly explicit local description of K\"ahler manifolds admiting 
a Hamiltonian 2-form. Using also Proposition \ref{prop:Feqn} one thereby obtains a \emph{local} classification of Sasaki-Einstein manifolds 
with a transverse K\"ahler structure that admits a transverse Hamiltonian 2-form. 
Via the comments in section \ref{sec:Ham}, this implies the existence of a transverse 
Killing tensor.
However, only in real dimension 5 have the global properties been investigated in detail, leading 
to Theorem \ref{thm:Ypq} and Theorem \ref{thm:Labc}. 
\begin{problem}
Classify all complete Sasaki-Einstein manifolds admitting a transverse Hamiltonian 2-form.
\end{problem}

\bibliographystyle{amsalpha}

\end{document}